\documentclass[11pt]{amsart}
\usepackage{amsmath, amsthm, amssymb,amscd}
\usepackage{float}
\usepackage{graphicx}
\usepackage{color}

\usepackage[all,cmtip]{xy}
\usepackage[english]{babel}
\newcommand{\seq}{\subseteq}
\newcommand{\C}{\mathbb{C}}

\newtheorem{thm}{Theorem}[section]
\newtheorem*{thm-nl}{Theorem}
\newtheorem*{prop-nl}{Proposition}
\newtheorem{definition}[thm]{Definition}
\newtheorem{lem}[thm]{Lemma}

\def\PP{{\textbf P}}
\def\OO{\mathcal{O}}

\def\X{\mathcal{X}}

\def\cM{\mathcal{M}}

\def\H{\mathcal{H}}
\def\Pic0{{\rm Pic}^0(X)}

\def\mm{\overline{\mathcal{M}}}

\def\hh{\overline{\mathcal{H}}}
\def\ph{\widetilde{\mathcal{G}}_{2a-1,a}^{\mathrm{ns}}}

\newtheorem*{cor-nl}{Corollary}
\newtheorem{conjecture}[thm]{Conjecture}
\newtheorem*{conjecture-nl}{Conjecture}
\newtheorem*{quest-nl}{Question}
\newtheorem*{quests-nl}{Questions}

\newtheorem{prop}[thm]{Proposition}

\theoremstyle{remark}

\newtheorem*{eg}{Example}
\newtheorem{remark}[thm]{Remark}

\title{{Linear Syzygies of Curves with prescribed gonality}}

\author[G. Farkas]{Gavril Farkas}
\address{Humboldt-Universit\"at zu Berlin, Institut f\"ur Mathematik,  Unter den Linden 6
\hfill \newline\texttt{}
 \indent 10099 Berlin, Germany} \email{{\tt farkas@math.hu-berlin.de}}

\author[M. Kemeny]{Michael Kemeny}

\address{Stanford University, Department of Mathematics, 450 Serra Mall
\hfill \newline\texttt{}
 \indent CA 94305, USA} \email{{\tt michael.kemeny@gmail.com}}

\begin{document}

\begin{abstract}
We prove two statements concerning the linear strand of the minimal free resolution of a $k$-gonal curve $C$ of genus $g$. Firstly,
we show that a general curve $C$ of genus $g$ of non-maximal gonality $k\leq \frac{g+1}{2}$ satisfies Schreyer's Conjecture, that is, $b_{g-k,1}(C,\omega_C)=g-k$.
This statement goes beyond Green's Conjecture and predicts that all highest order linear syzygies in the canonical embedding of $C$ are determined by the
syzygies of the $(k-1)$-dimensional scroll containing $C$. Secondly, we prove an optimal effective version of the Gonality Conjecture for general $k$-gonal curves, which makes more precise the (asymptotic) Gonality Conjecture proved by Ein-Lazarsfeld and improves results of Rathmann.

\end{abstract}

\maketitle
\setcounter{section}{-1}
\section{Introduction}

\noindent {\bf{1. The effective gonality conjecture.}}
Let $C$ be a smooth complex algebraic curve and $L$ a very ample line bundle on $C$ inducing an embedding $\varphi_L:C\hookrightarrow \PP H^0(C,L)$. In order to describe the equations of this embedding, after setting $r:=r(L)$,
we consider the finitely generated graded $S:=\text{Sym}\; H^0(C,L) \cong \C [x_0, \ldots, x_r]$-module $\Gamma_C(L):=\bigoplus_{n} H^0(C,L^{\otimes n}).$
By the \emph{Hilbert Syzygy Theorem}, one has a minimal free resolution
$$ 0  \longrightarrow{} F_{r-1}\longrightarrow  \cdots \longrightarrow F_0 \longrightarrow \Gamma_C(L) \longrightarrow 0,$$
where
$$ F_p= \bigoplus_{q>0} K_{p,q}(C,L) \otimes S(-p-q),$$
with $K_{p,q}(C,L)$ being the Koszul cohomology group of $p$-th order syzygies of weight $q$. As usual,  the \emph{graded Betti numbers} of $(C,L)$ are defined by
$ b_{p,q}:=\dim K_{p,q}(C,L).$ If $L$ is non-special, then $K_{p,q}(C,L)=0$ for all $q\geq 3$. Accordingly, the graded Betti diagram of $(C,L)$ consists only
of two non-trivial rows: the linear strand ($q=1$) and the quadratic strand ($q=2$).

The quadratic strand of the resolution is the subject of the Green-Lazarsfeld Secant Conjecture \cite{green-lazarsfeld-projective} and has been studied extensively in \cite{generic-secant}, \cite{extremal-gonality}. The linear row is the subject of the Gonality Conjecture formulated in the same paper \cite{green-lazarsfeld-projective}.

\vskip 3pt

Assume $C$ is $k$-gonal and let $L$ be a line bundle on $C$ of degree $\mbox{deg}(L)\geq 2g-1+k$. By the Green-Lazarsfeld Nonvanishing Theorem \cite[Appendix]{green-koszul}, one has $K_{h^0(L)-k-1,1}(C,L)\neq 0$. In a major breakthrough, generalizing results in \cite{AV} in the case of general $k$-gonal curves, Ein and Lazarsfeld \cite{ein-lazarsfeld} proved that for an \emph{arbitrary} smooth curve $C$ of gonality $k$, if
$\mbox{deg}(L)\gg 0$, then
\begin{equation}\label{goncon1}K_{h^0(L)-k,1}(C,L)=0.
\end{equation}
This result has been significantly improved by Rathmann \cite{rathmann}, who showed that the vanishing (\ref{goncon1}) holds for every smooth curve $C$ of genus $g$,
when $\mbox{deg}(L)\geq 4g-3$.  As already indicated in the original paper \cite{green-lazarsfeld-projective} Conjecture 3.7, one can ask for an effective version of the Gonality Conjecture. We show the following:

\begin{thm} \label{arb-gon}
Let $C$ be a general $k$-gonal curve of genus $g\geq 4$. Then for each line bundle  $L$ on $C$ of degree $\mathrm{deg}(L) \geq 2g-1+k$, one has
$$K_{h^0(L)-k,1}(C,L)=0.$$
\end{thm}

While the original Gonality Conjecture has been formulated as an asymptotic statement in $\mbox{deg}(L)$, the bound appearing in Theorem \ref{arb-gon} is already raised as a possibility in
\cite[page  86]{green-lazarsfeld-projective}.
Clearly Theorem \ref{arb-gon} implies $K_{p,1}(C,L)=0$, for all $p\geq h^0(C,L)-k$. The bound on $\mbox{deg}(L)$ appearing in Theorem \ref{arb-gon} is optimal. Indeed, if $A \in W^1_k(C)$ is a pencil of minimal degree, then $$K_{g-1,1}(C,\omega_C \otimes A) \neq 0,$$ by the Green--Lazarsfeld Nonvanishing Theorem, that is, on every curve there exist line bundles of degree $2g-2+k$ which do not verify (\ref{goncon1}).

\vskip 3pt

In the interest of convenience, we say that a smooth curve $C$ of genus $g$ and gonality $k$ satisfies the \emph{Effective Gonality Conjecture} if for each line bundle $L\in \mbox{Pic}^d(C)$, where $d\geq 2g-1+k$, one has $K_{h^0(L)-k,1}(C,L)=0$. Equivalently, if  there exists a line bundle $L\in \mbox{Pic}^{2g-1+k}(C)$ such that $K_{g,1}(C,L)\neq 0$, then $\mbox{gon}(C)\leq k-1$. Theorem \ref{arb-gon} can be reformulated as stating that a general $k$-gonal curve of genus $g\geq 4$ verifies the Effective Gonality Conjecture.

\vskip 3pt

By Green's $K_{p,1}$-theorem,
see \cite[Theorem 3.c.1]{green-koszul}, an \emph{arbitrary} $3$-gonal curve of genus $g\geq 4$ satisfies the Effective Gonality Conjecture. The same conclusion holds for each $4$-gonal curve of genus $g\geq 7$, see \cite[Proposition 3.8]{Te} or
\cite{aprodu-sernesi}. Note that Theorem \ref{arb-gon} fails for $g=3$. In this case, the general curve is trigonal and it is easy to see that $K_{3,1}(C,\omega_C^{\otimes 2}) \neq 0$, using the fact
that the canonical linear system embeds $C$ in the plane.

\vskip 4pt

For curves of maximal gonality of odd genus $g\geq 5$, our results are complete:

\begin{thm} \label{odd}
Every smooth curve of odd genus $g\geq 5$  and maximal gonality satisfies the Effective Gonality Conjecture.
\end{thm}

Theorem \ref{odd}, which plays an essential role in the proof of Theorem \ref{arb-gon} turns out to be intimately related to the divisorial case of the Green-Lazarsfeld Secant Conjecture proved in full generality \cite[Theorem 1.4]{generic-secant}. We observe that using \cite{generic-secant}, if $C$ is a smooth curve of genus $g=2n+1$ and gonality $n+2$, the following equivalence holds for a line bundle $M\in \mbox{Pic}^{2g}(C)$:
\begin{equation}\label{equiv}
K_{n,1}(C,M)\neq 0\Longleftrightarrow M-K_C\in C_{n+1}-C_{n-1}.
\end{equation}

The right hand side denotes the divisorial difference variety $C_{n+1}-C_{n-1}\subseteq \mbox{Pic}^2(C)$. An argument involving the geometry of secant varieties for line bundles on $C$ then shows that (\ref{equiv}) implies the vanishing
$K_{g,1}(C,L)=0$, for \emph{every} line bundle $L\in \mbox{Pic}^{5n+3}(C)$, thus establishing Theorem \ref{odd}. In order to deduce Theorem \ref{arb-gon}, we fix a value for the gonality $k\leq \frac{g+3}{2}$ and perform induction on the genus $g$; the initial step is Theorem \ref{odd}. By induction, assume that the general smooth curve $C$ of genus $g$ and gonality $k$ satisfies the Effective Gonality Conjecture. The stable curve $X$ of genus $g+1$ obtained by adding an elliptic curve $E$ at a point of ramification of a degree $k$ pencil on $C$ lies in the limit in $\mm_{g+1}$ of the locus of smooth $k$-gonal curves of genus $g+1$. An analysis of syzygies of line bundles of bidegree $(2g+k,1)$ on $X$ allows us to deduce the Effective Gonality Conjecture for a smooth deformation of $X$ having gonality $k$.

\vskip 5pt

\noindent {\bf{2. Schreyer's Conjecture}.}
Consider a general $k$-gonal curve canonically embedded curve $C\hookrightarrow \PP^{g-1}$ of gonality $k$. Green's Conjecture,  known in this case, see
\cite{V1}, \cite{V2}, \cite{aprodu-remarks}, and asserting that
$$K_{p,1}(C,\omega_C)=0 \ \ \mbox{ if and only if } \ \ \  p\geq g-k+1,$$
determines the length of the linear (as well as that of the quadratic) strand of the resolution of $C$. Schreyer's Conjecture \cite[\S 6]{schreyer-topics} and \cite{SSW} addresses the more refined question of what actually \emph{is} the Betti diagram of $C$, that is, determine the values $b_{p,1}(C,\omega_C)$ for $k-2\leq p\leq g-k$. Note that in the case when $C$ has the same gonality as a general curve of genus $g$, that is,  $\mbox{gon}(C)=\lfloor \frac{g+3}{2}\rfloor$, and only in this case, Green's Conjecture determines the entire resolution of $C$. Indeed, in this case Green's Conjecture is equivalent to the statement  that the resolution of $C\subseteq \PP^{g-1}$ is \emph{natural}, or equivalently $$b_{p,2}(C,\omega_C)\cdot b_{p+1,1}(C,\omega_C)=0$$ for all $p$. Since the differences $b_{p+1,1}(C,\omega_C)-b_{p,2}(C,\omega_C)$ are known and independent of $C$, knowing which Betti numbers vanish amounts to knowing the entire Betti diagram.

\vskip 3pt

Assume now $\mbox{gon}(C)\leq \frac{g+1}{2}$, that is, $C$ has non-maximal gonality.  In this case, Green's Conjecture predicts the following resolution, where we observe that $b_{p,1}(C,\omega_C)\cdot b_{p,2}(C,\omega_C)\neq 0$ for $k-2\leq p\leq g-k$.
\begin{table}[htp!]
\begin{center}
\begin{tabular}{|c|c|c|c|c|c|c|c|c|c|}
\hline
$1$ & $2$ & $\ldots$ & $k-3$ & $k-2$  & $\ldots$ & $g-k$ & $g-k+1$ & $\ldots$ & $g-2$\\
\hline
$b_{1,1}$ & $b_{2,1}$ & $\ldots$ & $b_{k-3,1}$ & $b_{k-2,1}$ &  $\ldots$ & $b_{g-k,1}$ & $0$ &  $\ldots$ & $0$ \\
\hline
$0$ &  $0$ & $\ldots$ & $0$ & $b_{k-2,2}$ &  $\ldots$ & $b_{g-k,2}$ & $b_{g-k+1,2}$ & $\ldots$ & $b_{g-2,2}$\\
\hline
\end{tabular}
\end{center}

    \caption{The Betti table of a general canonical $k$-gonal curve of genus $g$.}
    \label{tab:even}
\end{table}

It is known \cite{arbarello-cornalba} that such a curve $C$ carries a unique pencil $A\in W^1_k(C)$ of minimal degree, inducing a $(k-1)$-dimensional scroll $X\subseteq \PP^{g-1}$
swept out by the fibres of $|A|$. The Betti numbers of $(X, \OO_X(1))$ are determined by the Eagon-Northcott complex, see \cite{schreyer1}. Since $C\subseteq X\subseteq \PP^{g-1}$, one has
the following inequality (see also Section \ref{extending})
\begin{equation}\label{gonericconj}
b_{p,1}(C,\omega_C)\geq b_{p,1}\bigl(X,\OO_X(1)\bigr)=p\cdot {g-k+1 \choose p+1}.
\end{equation}
It was originally expected that the inequality (\ref{gonericconj}) is always an equality for $p \geq \lceil \frac{g-1}{2} \rceil$. This, however, is now known to fail. Indeed, Bopp \cite{bopp} showed that the for a general $5$-gonal curve of sufficiently high genus, if  $m:=\lceil \frac{g-1}{2}\rceil$, then $b_{m,1}(C,\omega_C)>b_{m,1}(X, \OO_X(1))$. Schreyer's Conjecture \cite{SSW} concerns the value of the highest non-zero
Betti number in the linear strand and predicts that in this case, under suitable generality assumptions, inequality (\ref{gonericconj}) is an equality.

\vskip 3pt

\begin{conjecture}[Schreyer's Conjecture]\label{schrconj1}
Let $C$ be a curve of genus $g$ and non-maximal gonality $3\leq k \leq \frac{g+1}{2}$. Assume $W^1_k(C)=\{A\}$ is a reduced single point and $A$ is the unique line bundle of degree at
most $g-1$ achieving the Clifford index. Then
$$b_{g-k,1}(C,\omega_C)=g-k.
$$
\end{conjecture}

\vskip 4pt

The converse statement is straightforward. Indeed, if $W^1_k(C)$ does \emph{not} consist of a reduced single point, then $b_{g-k,1}(C,\omega_C)>g-k$,  see
\cite[Proposition\ 4.10]{SSW}.
The condition $b_{g-k,1}(C,\omega_C)=g-k$ automatically implies the vanishing statements $b_{p,1}(C,\omega_C)=0$, for $p>g-k$.  Conjecture \ref{schrconj1} is known to hold for
a \emph{general} $k$-gonal curve provided $(k-1)^2<g$, see \cite{sch}. An important piece of evidence for the conjecture is the case of general $k$-gonal curves of odd genus $2k-1$.
Such curves form a divisor $\mathfrak{Hur}$ in the moduli space $\cM_{2k-1}$, much studied by Harris and Mumford in \cite{ha-mu}. Combining  results in \cite{hirsch} and those in \cite{V2}, it follows that
Conjecture \ref{schrconj1} holds in this case. Outside this divisorial range, little has been known. The main result of this paper is the following:
\begin{thm} \label{goneric0}
Schreyer's Conjecture holds for a general $k$-gonal curve $C$ of genus $g\geq 2k-1$:
$$b_{g-k,1}(C,\omega_C)=g-k.$$
\end{thm}

In fact we give a Brill-Noether theoretic sufficient condition for Schreyer's conjecture to hold for a given curve. In what follows, $G_d^{1, \mathrm{bpf}}(C)\seq G^1_d(C)$ denotes the subvariety of base point free pencils of degree $d$ on a curve $C$. We establish the following result implying Theorem \ref{goneric0}.

\begin{thm} \label{goneric}
Assume $C$ is a $k$-gonal curve $C$ of genus $g\geq 2k-1$ satisfying bpf-linear growth:
\begin{align*}
\dim G^1_{k+m}(C) &\leq m,  \;\; \text{for $0 \leq m \leq g-2k+1$} \\
\text{and \;}
\dim G^{1,\mathrm{bpf}}_{k+m}(C) &<m, \;\; \text{for $0 < m \leq g-2k+1$} .\end{align*}
Assume there is a unique pencil in $A \in G^1_k(C)$ having simple ramification and with $h^0(C, A^{\otimes 2})=3$. Then Schreyer's Conjecture holds for $C$.
\end{thm}

\vskip 4pt

Part of Theorem \ref{goneric} is that there is a canonical identification
$$K_{g-k,1}(C,\omega_C)\cong \bigwedge^{g-k+1} H^0(C,K_C\otimes A^{\vee})\otimes \mbox{Sym}^{g-k-1} H^0(C,A)\otimes \bigwedge^2 H^0(C,A),$$
where $A$ is the unique degree $k$ pencil on $C$.  \emph{All} the $(g-k)$-th syzygies linear syzygies of the canonical curve $C\seq \PP^{g-1}$ are of \emph{Eagon-Northcott type} and
can be written down explicitly. Precisely, if $(\tau_0,\ldots, \tau_{g-k})$ is a basis of $H^0(C,\omega_C\otimes A^{\vee})$ and $\sigma\in H^0(C,A)$, then the syzygy corresponding to the power
$\sigma^{g-k-1}\in
\mbox{Sym}^{g-k-1} H^0(C,A)$ has the form
$$\sum_{j=0}^{g-k}(-1)^j (\sigma \tau_1)\wedge \ldots \widehat{(\sigma\tau_j)}\wedge \ldots \wedge (\sigma\tau_{g-k})\wedge\bigl\{(\sigma\tau_0)\otimes (\sigma'\tau_j)-
(\sigma'\tau_0)\otimes (\sigma\tau_j)\bigr\}\in \bigwedge^{g-k} H^0(\omega_C)\otimes H^0(\omega_C),$$
where $\sigma'\in H^0(C,A)$ is another section such that $(\sigma,\sigma')$ form a basis of $H^0(C,A)$.

\vskip 4pt

It is tempting to interpolate between and link the two main results of this paper, namely Theorems \ref{arb-gon} and \ref{goneric}, and conjecture that a statement analogous
to Schreyer's Conjecture holds not only for the canonical bundle, but for every sufficiently positive line bundle on $C$. We fix a general $k$-gonal curve $C$ of genus $g\geq 2k-1$ and a line bundle $L$
on $C$ with $\mbox{deg}(L)\geq 2g+k$.
\begin{conjecture}\label{conj3}
If $r=r(L)$, one has $\mathrm{dim}\ K_{r-k,1}(C,L)=r-k.$
\end{conjecture}
We expect that all syzygies in $K_{r-k,1}(C,L)$ are again of Eagon-Northcott type, being induced by the $k$-dimensional scroll induced by the unique pencil $A\in W^1_k(C)$ and which contains the embedded curve $\varphi_L:C\hookrightarrow \PP^{r}$.

\vskip 5pt

The proof of Theorem \ref{goneric} begins in Section \ref{gonericity} with the already mentioned observation that via  \cite{hirsch} and \cite{V2}, a smooth curve $C$ of genus $2k-1$
and gonality $k$ satisfies $b_{k-1,1}(C,\omega_C)=k-1$, provided $W^1_k(C)$ is integral of dimension zero. Consider the Hurwitz space $\mathcal{H}_{2k-1,k}$ of smooth curves of
genus $g$ which are $k$-fold covers of $\PP^1$. We define the \emph{Eagon-Northcott} divisor $\mathcal{EN}$ on $\H_{2k-1,k}$ parametrizing moduli points
$[f:C\rightarrow \PP^1]\in \H_{2k-1,k}$ with
$b_{k-1,1}(C,\omega_C)>k-1$. In other words, points of $\mathcal{EN}$ correspond to canonical curves $C\seq \PP^{g-1}$ having a $(g-k)$-th order linear syzygy which is \emph{not} of
Eagon-Northcott type. We also consider the Brill-Noether type divisor $\mathfrak{BN}$ on $\H_{2k-1,k}$ consisting of points $[f:C\rightarrow \PP^1]$ such that $C$ has an extra pencil of
degree $k$. By the above discussion these two divisors coincide set-theoretically, that is,
$$\mathcal{EN}=\mathfrak{BN}.$$

\vskip 4pt

Now suppose we are no longer in the divisorial case and choose $k\leq \frac{g+1}{2}$. We follow a strategy reminiscent of \cite{aprodu-remarks}.
Starting with a general $k$-gonal curve $C$ of genus $g$, we form the irreducible nodal curve $[D] \in \mm_{2g-2k+1}$ obtained by identifying $g-2k+1$ general pairs of points on $C$.
Clearly $p_a(D)=2g-2k+1$ and $\mbox{gon}(D)\leq g-k+1$, that is, $[D]$ belongs to the closure $\overline{\mathfrak{Hur}}=\mm_{2g-2k+1,g-k+1}^1$ of the Hurwitz divisor, already considered
in \cite{ha-mu}, \cite{hirsch} and \cite{generic-secant}. Let
 $$\pi:\hh_{2g-2k+1,g-k+1}\rightarrow \mm_{2g-2k+1}$$
denote the forgetful map from the space of admissible covers of degree $g-k+1$ compactifying the Hurwitz space $\mathcal{H}_{2g-2k+1,g-k+1}$.
Assuming the curve $C$ we started with is sufficiently general, one checks directly that \emph{set-theoretically} $W^{1}_{g-k+1}(D)$ consists of one point
(that is, $\pi^{-1}([D])$ consists of one admissible cover $[f]$). This point corresponds to the torsion free sheaf on $D$ given by pushing forward the unique degree $k$ pencil on $C$.
In the last section of the paper we show that $[f]\notin \overline{\mathfrak{BN}}$, therefore $[f]\notin \overline{\mathcal{EN}}$.
To conclude $b_{g-k,1}(C,\omega_C)=g-k$, we extend in Section \ref{extending} the determinantal structure of the Eagon-Northcott divisor $\mathcal{EN}$ over a partial compactification of $\H_{2g-2k+1,g-k+1}$
containing the moduli point of $[f]$. In Section \ref{k3}, we then use $K3$ surfaces to show that this extended Eagon-Northcott divisor
does not contain the unique boundary component of $\hh_{2g-2k+1,g-k+1}$ containing $[f]$. Since always
$K_{g-k,1}(C,\omega_C)\hookrightarrow K_{g-k,1}(D, \omega_D)$, this completes the proof of Theorem \ref{goneric}.\footnote{It might be tempting to carry out
this argument at the level of $\mm_{2g-2k+1}$ rather than pass to Hurwitz space. However, the scheme structure of $W^1_{g-k+1}(D)$ is difficult to analyse,
so that $[D]$ may \emph{a priori} be a singular point of
$\mbox{Im}(\pi)=\overline{\mathfrak{Hur}}$. Thus a degenerate version of results in \cite{hirsch}, does not actually lead to a proof of Conjecture \ref{schrconj1}.}

\vskip 4pt

The organisation of the paper is as follows: We first review some background on syzygies of curves in Section \ref{background}. In Section \ref{gonality}, we prove Theorem
\ref{arb-gon}.  We prove Theorem \ref{goneric} in Sections \ref{gonericity}, \ref{extending} and \ref{k3}.

\vskip 3pt

\noindent {\bf Acknowledgments:} We thank Marian Aprodu, Wouter Castryck and Michael Hoff for stimulating conversations related to this circle of ideas. The second author thanks Christian Bopp for bringing Schreyer's conjecture and \cite{SSW}  to his attention, as well as for discussions on the (virtual) Koszul divisor of \cite{bopp}. Above all, we are grateful to Frank-Olaf Schreyer
for generously sharing with us his thoughts concerning his Conjecture \ref{schrconj1}. In particular, the idea of considering the Eagon-Northcott divisor $\mathcal{EN}$, important in the proof
of Theorem \ref{goneric}, is due to him. This work was supported by the DFG Priority Program 1489
\emph{Algorithmische Methoden in Algebra, Geometrie und Zahlentheorie}. The second author was also partially supported by NSF grant DMS-1701245.

\section{Background on Syzygies} \label{background}

We recall a few definitions and collect some basic results on syzygies that will be used throughout the paper.
Let $X$ be a (possibly singular) projective variety and let $L,M \in \text{Pic}(X)$ be line bundles. Consider the graded $S:=\text{Sym}\; H^0(X,L)$-module
$$\Gamma_{X}(M,L):=\bigoplus_{n \in \mathbb{Z}_{\geq 0}} H^0(X,L^{\otimes n}\otimes M). $$

One defines the Koszul cohomology groups $K_{p,q}(X,M;L)$ of $p$-th syzygies of weight $q$  by resolving the module $\Gamma_X(M,L)$ and computes them via the Koszul complex, see \cite{green-koszul}. When $M=\mathcal{O}_X$, we write $K_{p,q}(X,L):=K_{p,q}(X,\OO_X;L)$. The following fact is surely well-known:

\begin{lem}[Semicontinuity] \label{semi}
Let $\pi:\; \mathcal{X} \to S$ be a flat, projective morphism of schemes over an integral base. Let $\mathcal{L} \in \mathrm{Pic}(\mathcal{X})$ be a line bundle such that $h^0(X_s, \mathcal{L}_s)=c$, for each $s \in S$. Let $\mathcal{M} \in \mathrm{Pic}(\mathcal{X})$ be a second line bundle, and assume $$h^0(X_s, \mathcal{L}_s^{\otimes (q-1)} \otimes \mathcal{M}_s )=r_{1},\; h^0(X_s, \mathcal{L}_s^{\otimes q}\otimes \mathcal{M}_s )=r_{2},\; h^0(X_s, \mathcal{L}_s^{\otimes (q+1)}\otimes \mathcal{M}_s )=r_{3}$$ are also independent of $s \in S$. Then the function
\begin{align*}
\psi : s &\mapsto \dim K_{p,q}(X_s, \mathcal{M}_s; \mathcal{L}_s)
\end{align*}
is upper semicontinuous on $S$.
\end{lem}

We collect  some results on syzygies of curves which, taken together, reduce Theorem  \ref{arb-gon} to the extremal case of line bundles of degree $d=2g-1+\mbox{gon}(C)$. We quote from \cite{aprodu-nagel}, Theorem 4.27:
\begin{lem}\label{lem1}
Let $C$ be a smooth curve of genus $g$ and $L$ a line bundle of degree $d \geq g$ with $h^1(C, L)=0$. Assume $K_{p,1}(C,L)=0$. Then $K_{p+1,1}(C,L(x))=0$, for any point $x \in C$ such that $L(x)$ is base point free.
\end{lem}

It is standard, see e.g. \cite{aprodu-nagel}, Corollary 2.13, that if  $L\ncong \OO_C$ is a globally generated line bundle on  a smooth curve $C$, if $K_{p,1}(C,L)=0$, then $K_{p+1,1}(C,L)=0$. Accordingly, there are several natural invariants which one can read directly off the Betti table of an embedded curve $C\stackrel{|L|}\hookrightarrow \PP^{r(L)}$ and which measure the length of the linear and the quadratic strand respectively:
$$\hspace{1.3cm} \ell_1(C,L):=\mbox{max} \bigl\{p\in \mathbb N_{>0}: b_{p,1}(C,L)\neq 0 \bigr\} \ \ \ \mbox{ and }$$
$$\ell_2(C,L):=\mbox{min} \bigl \{p\in \mathbb N_{>0}: b_{p,2}(C, L)\neq 0 \bigr\}.$$

Recalling that $K_{p,q}(C,L)=0$ for $p\geq r(L)$, the invariants $\ell_1(C,L)$ are encoded in the more classical properties $(N_p)$ and $(M_q)$ defined in \cite{green-lazarsfeld-projective}.
Precisely, $\ell_2(C,L)$ is the smallest integer such that $(C,L)$ fails property $\bigl(N_{\ell_2(C,L)}\bigr)$, whereas $\ell_1(C,L)$ is the smallest integer such that $L$ fails property $\bigl(M_{r(L)-\ell_1(C,L)}\bigr)$.

\vskip 3pt

\section{The Effective Gonality Conjecture for generic curves} \label{gonality}
We start by proving Theorem \ref{odd}. It turns out that our proof of the generic Green-Lazarsfeld Secant Conjecture \cite{generic-secant} takes us a long distance towards finding a complete solution.

\begin{proof}[Proof of Theorem \ref{odd}]
Let $C$ be a curve of genus $2n+1$ and gonality $n+2$. Then using e.g. \cite[Remark 6.3]{hirsch}, we observe that $\mbox{Cliff}(C)=n$, that is, $C$ has maximal Clifford index as well. We need to prove that for any line bundle $L\in \mbox{Pic}(C)$ of degree at least $5n+3$, we have $K_{i,1}(C,L)=0$ for $i \geq h^0(C, L)-n-2$. We may assume $n \geq 2$ and as explained in the previous section, it is enough to prove that for any line bundle $L\in \mbox{Pic}^{5n+3}(C)$, we have $K_{2n+1,1}(C,L)=0$.

\vskip 3pt

Theorem 1.4 of \cite{generic-secant} establishes the following equivalence for any line bundle $M\in \mbox{Pic}^{4n+2}(C)$:
$$ K_{n-1,2}(C,M) \neq 0 \Longleftrightarrow M-K_C \in C_{n+1}-C_{n-1}.$$  For any line bundle $M \in \text{Pic}^{4n+2}(C)$ one has cf. \cite[formula (8)]{generic-secant}
$$\dim K_{n,1}(C,M)=\dim K_{n-1,2}(C,M).$$ Thus, for any $M \in \text{Pic}^{4n+2}(C)$, the equivalence
$$ K_{n,1}(C,M) \neq 0 \Longleftrightarrow M-K_C \in C_{n+1}-C_{n-1}$$
holds. Using Lemma \ref{lem1} again, it thus suffices to show that for any line bundle $L$ of degree $5n+3$, there exists an effective divisor $D \in C_{n+1}$ such that
$$ L-D-K_C \notin C_{n+1}-C_{n-1}.$$ Suppose this were not the case, that is,
$$L-K_C-C_{n+1} \seq  C_{n+1}-C_{n-1}.$$ Then for every $D \in C_{n+1}$ there exists a divisor $E \in C_{n+1}$ such that
$H^1(C,L(-D-E)) \neq 0$, that is, $D+E$ is an element of the (determinantal) secant variety $V^{2n+1}_{2n+2}(L)$ of effective divisors failing to impose independent conditions on $|L|$. In particular,
$$ \dim V^{2n+1}_{2n+2}(L) \geq n+1,$$
which is one higher than the expected dimension $n$. We observe that  the morphism
\begin{align*}
\psi : \; V^{2n+1}_{2n+2}(L) & \to C_{n-1}, \\
A &\mapsto K_C-L+A
\end{align*}
is well-defined, since $h^0(C,K_C-L+A)=1$, for $\mbox{gon}(C)>n-1$.  Let $I$ be any component of $V^{2n+1}_{2n+2}(L)$ of dimension $n+1$ and set $r:=n-1-\dim \ \psi(I)$. Then $\psi_{|_I}$ must have fibres of dimension at least $2+r$.
As all divisors in the inverse image $\psi^{-1}(B)$ are clearly linearly equivalent, we have $h^0(C, A)\geq 3+r$ for
all $A \in V^{2n+1}_{2n+2}(L)$ such that $\psi(A)=B \in \psi(I)$. By Riemann--Roch, this implies $h^1(C, A) \geq 1+r$, or $h^0(C, K_C-A)=h^0(2K_C-L-B)\geq 1+r$. The latter inequality holds for \emph{any} effective divisor $B \in \psi(I)$, so we must have
$$\dim |2K_C-L| \geq r+\dim \psi(I)=n-1.$$ This implies $h^1(C,2K_C-L) \geq 3$, or equivalently $L-K_C \in W^2_{n+3}(C)$. But then $\text{Cliff}(C) \leq n-1$ (if $n=2$, then compute the Clifford index of $2K_C-L$ rather than $L-K_C$). Since we have $\text{Cliff}(C)=n$, this is a contradiction.
\end{proof}

The proof of  Theorem \ref{odd} gives a characterisation of those line bundles $L\in \mbox{Pic}^{2g-2+\mathrm{gon}(C)}(C)$, such that  $K_{h^0(L)-\mathrm{gon}(C),1}(C,L)\neq 0$, in
the case where $C$ has odd genus and maximal gonality.
\begin{prop} \label{failure}
Let $C$ be a smooth curve of odd genus $2n+1$ and gonality $n+2$. Let $L \in \mathrm{Pic}^{5n+2}(C)$ be such that $K_{2n,1}(C,L) \neq 0$.
Then $L-K_C \in W^1_{n+2}(C).$
\end{prop}
\begin{proof}
Following the proof of Theorem \ref{odd}, we obtain $\dim V^{2n}_{2n+1}(L) \geq n$. By studying the morphism
$$
\psi: V^{2n}_{2n+1}(L) \to C_{n-1}, \ \  \
A \mapsto K_C-L+A.
$$
and arguing as in Theorem \ref{odd}, we are again led to the statement $h^0(C, 2K_C-L) \geq n$. The Riemann--Roch theorem gives $h^0(C, L-K_C) \geq 2$, as required.
\end{proof}

We shall prove Theorem \ref{arb-gon} by induction on the genus, fixing the gonality. To perform the induction step, let $C$ be a smooth
genus $g$ curve of gonality $k$ and denote by $f:C \to \PP^1$ the induced degree $k$ cover. We assume that $C$ verifies the Effective Gonality Conjecture.
Let $p \in C$ be a branch point
of $f$, and consider the stable curve $X=C \cup_p E$ obtained by glueing a smooth, genus $1$ curve at $p$. A standard argument with admissible covers or limit linear series shows that $X$ is a limit of smooth $k$-gonal curves of genus $g+1$, see \cite[\S 3.G]{ha-mu}.
\begin{prop} \label{elliptic}
Let $X=C \cup_p E$ be the genus $g+1$ stable curve as above and $L$ a line bundle on $X$ with $\deg (L_{C})=2g+k$ and $\deg (L_{E})=1$. Then, for a general point $q \in E \setminus \{p\}$, we have $$K_{g,1}(X,L(-q))=0.$$ Further, for such a point, $h^1(X,L(-q))=h^1(X,L^{\otimes 2}(-2q))=0$.
\end{prop}
\begin{proof}
We have the Mayer-Vietoris sequence on $X$
$$0 \longrightarrow L_{C}(-p) \longrightarrow L(-q) \longrightarrow L_{E}(-q) \longrightarrow 0.$$
For a general point $q \in E \setminus \{p\}$, we have $h^0(E,L_E^{\otimes j}(-jq))=h^1(E,L_E^{\otimes j}(-jq))=0$ for $j=1,2$, which implies $h^1(X,L(-q))=h^1(X,L^{\otimes 2}(-2q))=0$. Further, we have a natural isomorphism $H^0(C,L(-p) ) \cong H^0(X,L(-q))$, and we know, by the assumptions on $C$, that $$K_{g,1}(C,L(-p))=0.$$ We will use this to deduce $K_{g,1}(X,L(-q))=0$.

We have a natural commutative diagram
{\small{$$
\xymatrix@C=1em{
\bigwedge^{g+1} H^0(C,L(-p)) \ar[r]^-d \ar[d]^{\alpha}  & \bigwedge^{g} H^0(C,L(-p)) \otimes H^0(C,L(-p)) \ar[r]^-d \ar[d]^{\beta}& \bigwedge^{g-1} H^0(C, L(-p))\otimes H^0(C,L^{\otimes 2}(-2p)) \ar[d]^{\gamma} \\
\bigwedge^{g+1} H^0(X,L(-q)) \ar[r]^-d  & \bigwedge^{g} H^0(X,L(-q)) \otimes H^0(X,L(-q)) \ar[r]^-d &\bigwedge^{g-1} H^0(X, L(-q))\otimes H^0(X,L^{\otimes 2}(-2q))
},
$$}}

where $\alpha, \beta$ are isomorphisms, and $\gamma$ is induced from the natural composition
$$ H^0(C,L^{\otimes 2}(-2p)) \hookrightarrow H^0(C,L^{\otimes 2}(-p))\cong H^0(X,L^{\otimes 2}(-2q)).$$ As $K_{g,1}(C,L(-p))=0$, the top row is exact and since $\beta$ is surjective and $\gamma$ is injective, the bottom row must also be exact, as required.
\end{proof}

\vskip 4pt

From Proposition \ref{elliptic} we readily deduce Theorem \ref{arb-gon}, that is, the effective version of the Gonality Conjecture.
\begin{proof}
Fix $k \geq 4$. Assume that for the general $k$-gonal curve $C$ of genus $g$ one has $K_{g,1}(C,L)=0$, for any line bundle $L \in \text{Pic}^{2g-1+k}(C)$. We claim there exists a smooth curve $C'$ of genus $g+1$ and gonality $k$, such that $K_{g+1,1}(C',L')=0$,  for each line bundle $L' \in \text{Pic}^{2g+1+k}(C')$. By performing induction on $g$ and noting that the initial step is Theorem \ref{odd}, this suffices to prove the theorem. By Lemma \ref{lem1}, it further suffices to prove that there exists a smooth curve $C'$ of genus $g+1$ and gonality $k$ such that, for each line bundle $L' \in \text{Pic}^{2g+1+k}(C')$, there exists a point $q \in C'$ such that $K_{g,1}(C',L'(-q))=0$.

\vskip 3pt

Let $X=C \cup_p E$ be the genus $g+1$ stable curve introduced in Proposition \ref{elliptic}. Consider a flat family $\pi: \mathcal{C} \to \Delta$ of stable curves over a smooth, pointed, one dimensional base $(\Delta,0)$, such that the central fibre is $X$ and $\pi^{-1}(s)$ is a smooth curve of gonality $k$ for all $0 \neq s \in \Delta$. As $X$ is a curve of compact type, after shrinking $\Delta$ and performing a finite base change if necessary, we have a relative Picard scheme $$v:\;  \mathcal{P}ic^{2g+1+k}(\mathcal{C}/\Delta) \to \Delta,$$ with central fibre consisting of all line bundles of multidegree $(2g+k,1)$ on $X=C \cup_p E$; this scheme is flat and proper over $\Delta$, see \cite[\S 4]{deligne} and \cite{eisenbud-harris-limit}, proof of Theorem  3.3.

\vskip 3pt

Let $\mathcal{C}_0:=\mathcal{C} \setminus \{p\}$ be the open set of all points which are smooth in the fibres over $\Delta$. By Proposition \ref{elliptic} together with semicontinuity for the dimension of Koszul groups, there is an open subset $U \seq \mathcal{P}ic^{2g+1+k}(\mathcal{C}/\Delta) \times_{\Delta} \mathcal{C}_0$ such that for each pair $(L',q') \in U$, one has $K_{g,1}(C',L'(-q))=0$, where $C'=\pi^{-1}(v(L'))$, and such that $$0 \notin v\Bigl(\mathcal{P}ic^{2g+1+k}(\mathcal{C}/\Delta) \setminus \mbox{pr}_1(U)\Bigr),$$
where $\mbox{pr}_1:\mathcal{P}ic^{2g+1+k}(\mathcal{C}/\Delta) \times_{\Delta}\mathcal{C}_0 \to \mathcal{P}ic^{2g+1+k}(\mathcal{C}/\Delta)$ is the projection. As flat morphisms are open, $\mbox{pr}_1(U)$ is open, and since $v$ is proper, the image $$V:=v\bigl(\mathcal{P}ic^{2g+1+k}(\mathcal{C}/\Delta)\setminus \mbox{pr}_1(U)\bigr)$$is closed. Thus if $0 \neq t \in \Delta \setminus V$ and
$C_t:= \pi^{-1}(t)$, then, for each $L \in \text{Pic}^{2g+1+k}(C_t)$ there exists $q \in C_t$ with $K_{g,1}(C_t, L(-q))=0$, as required.

\end{proof}

\section{Schreyer's Conjecture for general curves of non-maximal gonality} \label{gonericity}
In this section, we begin discussing Schreyer's Conjecture for general $k$-gonal curves of genus $g\geq 2k-1$. We start by explaining the relevance of \cite{hirsch} for Conjecture
\ref{schrconj1}.

\vskip 3pt

For $g=2k-1$, we consider two divisors on $\mathcal{M}_g$, which already played a role in \cite{aprodu-remarks} or \cite{generic-secant}:
\begin{align*}
\mathfrak{Syz}:&= \bigl\{ [C] \in \mathcal{M}_g: K_{k-1,1}(C, \omega_C) \neq 0 \bigr\} \\
\mathfrak{Hur}:&= \bigl\{ [C] \in \mathcal{M}_g: W^1_{k}(C)\neq \emptyset\bigr\}.
\end{align*}
Recall that $\mathfrak{Syz}$ has a structure of degeneracy locus, whereas $\mathfrak{Hur}$ is the push-forward of the smooth Hurwitz space $\mathcal{H}_{2k-1,k}$ of degree $k$ covers of
$\PP^1$. We view both $\mathfrak{Syz}$ and $\mathfrak{Hur}$ as divisors on the \emph{moduli stack} of smooth curves of genus $2k-1$, rather than on the associated coarse moduli space.
It is proved in \cite{hirsch}, that one has the
following relation at stack level:
$$ [\mathfrak{Syz}] = (k-1)[\mathfrak{Hur}]\in CH^1(\cM_{2k-1}).$$
\begin{thm}\label{hirsch} (\cite{hirsch})
Let $C$ be a curve of genus $2k-1$ and gonality $k$ such that the point $W^1_k(C)$ consists of a reduced single point. Then $b_{k-1,1}(C,\omega_C)=k-1$.
\end{thm}
\begin{proof}
For a smooth curve $C$, we denote by $\phi:X\rightarrow (S,0)$ its versal deformation space, hence the associated moduli map $m(\phi):S\rightarrow \cM_g$ is an \'etale neighbourhood  of the
point $[C]\in \cM_g$.  For $s\in S$, set $C_s:=\phi^{-1}(s)$, thus $C_0=C$. From \cite{hirsch}, there exist two vector bundles $V$ and $W$ of the same rank over $S$ together with
a morphism $\chi:V \to W$,
such that, for any $s\in S$, we may identify $K_{k-1,1}(C_s,\omega_{C_s})=\mbox{Ker}(\chi_{s})$. Then the divisor $\mathfrak{Syz}(\phi)\seq S$ is defined by $\det(\chi)$.
Suppose $b_{k-1,1}(C, \omega_C) \geq k$. Thus $\det(\chi)$ vanishes to order at least $k$, cf.\ \cite[Lemma 6.1]{hirsch}. By the equality of cycles
$\mathfrak{Syz}(\phi) = (k-1)\mathfrak{Hur}(\phi)$ on $S$, the function defining $\mathfrak{Hur}(\phi)$ must vanish to order at least two. Thus $\mathfrak{Hur}(\phi)$ is not smooth at the point
$0\in S$. On the other hand it is well-known, see \cite{coppens}, that $\mathfrak{Hur}(\phi)$ is smooth at a point $0\in S$ corresponding to a curve $C$ if and only if $W^1_k(C)$ consists of a single pencil $A$ and, moreover, $h^0(C, A^{\otimes 2})=3$.
\end{proof}
\begin{remark} \label{HR-nodal}
One can generalise Theorem \ref{hirsch} as follows. For an integral nodal curve $D$,  we define $W^1_k(D)\subseteq \overline{\mbox{Pic}}^k(D)$ to be the closed subset of the compactified Jacobian of rank one, torsion free sheaves $A$ of degree $k$ on $C$ with $h^0(D,A) \geq 2$. Suppose $D$ is integral, nodal of genus $2k-1$ and assume $W^1_k(D)=\{A\}$, where $A$ is \emph{locally-free} and base-point free. Then the proof of Theorem \ref{hirsch} shows $b_{k-1,1}(D,\omega_D)=k-1$.
\end{remark}

\vskip 3pt

We now turn our attention to curves of genus $g$ and non-maximal gonality $k\leq \frac{g+1}{2}$. Let $G_d^{1, \mathrm{bpf}}(C)\seq G^1_d(C)$ be the subvariety of base point
free pencils of degree $d$ on $C$ and further let $W^1_d(C)$ denote the Brill--Noether variety of line bundles of degree $d$ with at least two sections. Note that there is a morphism $G^1_d(C) \to W^1_d(C)$, with fibre over a point  $[L] \in W^1_d(C)$ equal to the Grassmannian of pencils $V \seq H^0(C,L)$ . The following observation is a slight modification of the linear growth condition of \cite[Theorem 2]{aprodu-remarks}:

\begin{lem}\label{bpfgr}
A general curve $C$ of genus $g$ and gonality $k\leq \frac{g+1}{2}$ satisfies \emph{bpf-linear growth}:
\begin{align*}
\dim G^1_{k+m}(C) &\leq m,  \;\; \text{for $0 \leq m \leq g-2k+1$} \\
\text{and, further, \;}
\dim G^{1,\mathrm{bpf}}_{k+m}(C) &<m, \;\; \text{for $0 < m \leq g-2k+1$} .\end{align*}
\end{lem}
\begin{proof}

From \cite{aprodu-remarks}, we have $\mbox{dim } W^1_{k+m}(C)=m$, for  $0\leq  m \leq g-2k+1$.
We observe that if $Z \seq W^r_{d}(C)$ is an irreducible component, then $Z \cap W^{r+1}_d(C)$ has codimension at least two in $Z$, provided $g-r+d \geq 0$. This follows from the fact that no component of $C^r_d$ is entirely contained in $C^{r+1}_d$, where $C^r_d$ is the variety parametrizing divisors $D$ of degree $d$ on $C$ with $\dim |D| \geq r$, see \cite[\S IV.1]{ACGH1}.

We claim $\dim G^1_{d+m}(C) \leq m,$ for $0 \leq m \leq g-2k+1$. Take an irreducible component $J \seq G^1_{d+m}(C)$ and consider the restriction to $J$ of the surjection $c: G^1_{k+m}(C)\rightarrow W^1_{k+m}(C)$. Assume $c(J) \seq W^{1+j}_{d+m}(C)$ and choose $j\geq 0$ maximal with this property. Then by the above, $\dim \psi(J) \leq m-2j$. Since the general fibre of $c_{|J}$ is isomorphic to the Grassmannian $G(2,2+j)$, it follows $\dim J \leq 2j+\dim{c(J)}\leq m$. By an identical argument and using
\cite[Theorem 2.6]{arbarello-cornalba}, we similarly obtain that $\dim G^{1,\mathrm{bpf}}_{k+m}(C) <m$, in the range $0<m\leq g-2k+1$.
\end{proof}

The next Proposition is similar to Theorem 2 in \cite{aprodu-remarks} and we skip the details.
\begin{prop} \label{tf-limit}
Let $C$ be a smooth curve of genus $g$ and gonality $k\leq \frac{g+1}{2}$. Assume $C$ satisfies bpf-linear growth and $W^1_k(C)$ consists of a single point $A$. If $(x_i,y_i)$ are general pairs of points on $C$, where $1\leq i \leq g-2k+1$, let $D$ be the nodal curve obtained by glueing $x_i$ to $y_i$ for all $i$. Then  $W^1_{g-k+1}(D)=\{\nu_*(A)\}$, where $\nu : C \to D$ is the normalisation morphism. Furthermore, $\mathrm{gon}(D)=g-k+1$.
\end{prop}

\vskip 3pt

Consider the moduli space $\hh_{g,k}$ of degree $k$ admissible covers of genus $g$. Precisely, $$\hh_{g,k}=\mm_{0,2g+2k-2}(\mathcal B \mathfrak{S}_k)/\mathfrak{S}_{2g+2k-2}$$ is the space of twisted stable maps from genus zero curves into the classifying stack $\mathcal{B} \mathfrak S_k$ of the symmetric group $\mathfrak{S}_k$ and which are simply branched over $2g+2k-2$ points which we do not order. We refer to \cite{twisted} for the construction of this space. It is known that $\hh_{g,k}$ is the normalisation of the space of admissible covers constructed by Harris and Mumford in \cite{ha-mu}.  There is a morphism $\pi: \hh_{g,k} \to \overline{\mathcal{M}}_g$ given by stabilisation of the source curve of each admissible cover and then $\mbox{Im}(\pi)=\overline{\mathfrak{Hur}}$. The following result is the translation of Proposition \ref{tf-limit} to the moduli space of admissible covers.

\vskip 3pt

\begin{prop} \label{adm-limit}
Let $C$ be a smooth curve of genus $g$ and gonality $k\leq \frac{g+1}{2}$.  Assume $C$ satisfies bpf-linear growth and that $W^1_k(C)$ consists of a single point $A$, which we assume to have only simple ramification. For $1 \leq i \leq g-2k+1$, we choose general pairs of points $(x_i,y_i)$ on $C$ and let $[D]\in \mm_{2g-2k+1}$ be the nodal curve obtained by glueing $x_i$ to $y_i$. If
$$ \pi:  \hh_{2g-2k+1,g-k+1} \to \overline{\mathcal{M}}_{2g-2k+1}$$
is the forgetful map, then $\pi^{-1}([D])$ consists of a unique point $[f': B' \to T]$.
\end{prop}
\begin{proof}
We show that the construction described in \cite[Theorem\ 5]{ha-mu} is \emph{unique} in our case.
Let $[f': B' \to T]\in \hh_{2g-2k+1, g-k+1}$ be an admissible cover, where $p_a(T)=0$ and $B'$ is a nodal curve whose stable model is isomorphic to $D$. There exists a unique component $C_0$ of $B'$ having positive genus. The restriction $f_0:=f'_{|_{C_0}}$ gives a morphism $f_0:C_0 \to \PP^1_0$ onto a smooth rational component $\PP^1_0$ of $T$. By admissibility, $C_0 \cong C$ and $\deg(f_0) \geq k$.

Assume that $f_0(x_i)=f_0(y_i)$ if and only if $1 \leq i \leq j$. For $i=j+1,\ldots, g-2k+1$, we denote by $R_{x_i}$ and $R_{y_i}$ the irreducible components of $B'$ meeting $C$ at $x_i$ and $y_i$ respectively. As the stabilisation of $B'$ is $D$ and $f'(R_{x_i})\cap f'(R_{y_i})=\emptyset$, for each such $i$ there must be a component $\widetilde{R}_i$ of the subcurve $\overline{B'-C_0}$ of $B'$, such that $f'(\widetilde{R}_i)=\PP_0^1$, or else $T$ contains a loop. As $\deg(f')=g-k+1$, this implies that $d:=\deg(f_0) \leq k+j$.

\vskip 3pt

Since the pairs $(x_1, y_1), \ldots, (x_{j},y_{j})$ are general and $f_0$ gives rise to an element of $G^{1,\mathrm{bpf}}_d(C)$, it follows $\mbox{dim } G_d^{1,\mathrm{bpf}}(C)\geq j$. If $d>k$, this contradicts the bpf-linear condition on $C$, which implies that $\mbox{deg}(f_0)=k$ and $f_0$ is the map induced by the pencil of minimal degree $A \in W^1_k(C)$. Each $\widetilde{R}_i$ maps isomorphically onto $\PP^1_0$. Clearly $\mbox{deg}(f'_{R_{x_i}})\geq 2$ and $\mbox{deg}(f'_{R_{y_i}})\geq 2$, in particular $f'_{R_{x_i}}$ and $f'_{R_{y_i}}$ will both contain at least two ramification points of $f'$, for each $i=1, \ldots, g-2k+1$ (Note that being general points,  $x_i, y_i$ are not among the ramification points of $f_0$). Counting the total number of ramification points of
the cover $f'$, it follows that $\mbox{deg}(f'_{R_{x_i}})=\mbox{deg}(f'_{R_{y_i}})=2$. The morphism $f'$ is now uniquely determined, for $f'^{-1}(\PP^1_0)=C\cup \widetilde{R}_1\cup \ldots \cup \widetilde{R}_{g-2k+1}$ and all the components of $f'^{-1}(f(R_{x_i}))$ and $f'^{-1}(f(R_{y_i}))$ other than $R_{x_i}$ and $R_{y_i}$ respectively  are mapped isomorphically onto their images.
\end{proof}

\vskip 4pt
\begin{definition}
Let $\overline{\mathfrak{BN}}'\seq \hh_{2g-2k+1,g-k+1} \times_{\overline{\mathcal{M}}_{2g-2k+1}} \hh_{2g-2k+1,g-k+1} $
be the closure of the locus of pairs of covers $\bigl([g_1: C \to \PP^1], [g_2: C \to \PP^1]\bigr)$, where $C$ is a smooth curve of genus $2g-2k+1$  and $g_1\ncong g_2$.
We introduce the Brill-Noether divisor of curves possessing an extra pencil
$\overline{\mathfrak{BN}}:= \mathrm{pr}_1(\overline{\mathfrak{BN}}') \seq \hh_{2g-2k+1,g-k+1}$,
where $\mathrm{pr}_1$ denotes the first projection. We set $\mathfrak{BN}:=\overline{\mathfrak{BN}}\cap \mathcal{H}_{2g-2k+1,g-k+1}$.
\end{definition}

Applying \cite[Proposition 2.4]{arbarello-cornalba}, we know that
$\dim \overline{\mathfrak{BN}}' =\dim \mathcal{H}_{2g-2k+1,g-k+1}-1$. Since $\overline{\mathfrak{BN}}'$ is birational to the Severi variety of nodal curves of type $(g-k+1,g-k+1)$ on $\PP^1\times \PP^1$ having geometric genus $2g-2k+1$, using \cite{Ty},
we conclude that $\overline{\mathfrak{BN}}$ is an irreducible divisor. We also recall Coppens' result \cite{coppens} saying that if a  curve $[C]\in \cM_{2g-2k+1}$  has a pencil
$A\in W^1_{g-k+1}(C)$ with $h^0(C, A^{\otimes 2})\geq 4$, then $[C,A] \in \overline{\mathfrak{BN}}$. The locus of such pairs $[C,A]\in \mathcal{H}_{2g-2k+1,g-k+1}$ is of pure
codimension one in $\overline{\mathfrak{BN}}$.

\vskip 3pt

Our goal is to show that, in the notation of Proposition \ref{adm-limit}, the unique point of $\pi^{-1}([D])$ does not lie in
$\overline{\mathfrak{BN}}$ provided the normalisation $C$ is sufficiently general. To ease the notation, set $a:=g-k+1$ and assume $a \geq 3$. To carry out the argument,  it is convenient to work with stable maps rather than admissible covers. Let
$$  \widetilde{\mathcal{G}}_{2a-1,a}^{\mathrm{ns}}:= \widetilde{\mathcal{M}}_{2a-1}^{\mathrm{ns}}(\PP^1,a)$$
denote the moduli space of \emph{finite} stable maps $f: C \to \PP^1$ of degree $a$ such that $C$ has genus $2a-1$, has only \emph{non-separating} nodes and with
$h^0(C, f^*\mathcal{O}_{\PP^1}(1))=2$.
Then $\widetilde{\mathcal{G}}_{2a-1,a}^{\mathrm{ns}}$ is an open subset of the projective moduli space  $\mm_{2a-1}(\PP^1,a)$
of stable maps $f: C \to \PP^1$ with $f_*[C]=a[\PP^1]$.
Let $$\widetilde{\pi}: \; \widetilde{\mathcal{G}}_{2a-1,a}^{\mathrm{ns}}\to \mm_{2a-1}$$
denote the natural projection. The Hurwitz space $\mathcal{H}_{2a-1,a}$ can be realized as the quotient of an open set
of $\widetilde{\mathcal{G}}_{2a-1,a}^{\mathrm{ns}}$ by $PGL(2)$. We associate a stable map $[f: B \to \PP^1] \in \widetilde{\mathcal{G}}_{2a-1,a}^{\mathrm{ns}}$ to the unique point $[f' : B' \to T] \in \pi^{-1}([D])$ by letting $B$ be the curve obtained from $B'$ by contracting all components of $B'$ whose image is different from $f'(C)$, and then letting $[f: B\to \PP^1]$ be the map which, on each component of $B$, agrees with $f'$ on the corresponding component of $B'$ (this is only determined up to the $PGL(2)$ action). Then $$B=C \cup R_1\cup \ldots \cup R_{a-k},$$ where $R_i \cong \PP^1$ meets $C$ at two general points $(x_i,y_i)$ and $\deg(f_{R_i})=1$ for $i=1, \ldots, a-k$.

\vskip 4pt

Let $\mathfrak{B}^{\mathrm{ns}}_a \seq \widetilde{\mathcal{G}}_{2a-1,a}^{\mathrm{ns}} \times_{\mm_{2a-1}} \widetilde{\mathcal{G}}_{2a-1,a}^{\mathrm{ns}} $
be the closure of the locus of pairs $$\Bigl([g_1: X \to \PP^1], [g_2: X \to \PP^1]\Bigr),$$ where $X$ is a smooth curve of genus $2a-1$ and there is no automorphism $\sigma \in PGL(2)$ such that $[g_1] \cong \sigma \cdot [g_2]$. In order to prove that the unique point of $\pi^{-1}([D])$ does not lie in $\overline{\mathfrak{BN}}\subseteq \hh_{2a-1,a}$ it is sufficient to prove that $$([f],[f]) \notin \mathfrak{B}^{\mathrm{ns}}_a.$$

\vskip 3pt

For $0 \leq n \leq 2a-5$, let $\widetilde{\mathcal{M}}_{2a-1-n}^{\mathrm{ns}}(\PP^1,a-n;2n)$ denote the moduli space of finite stable maps $f: C \to \PP^1$ of degree $a-n$ with $2n$ markings and such that $C$ has genus $2a-1-n$, non-separating nodes and $h^0(C, f^*\mathcal{O}_{\PP^1}(1))=2$. Then $ \widetilde{\mathcal{M}}_{2a-1-n}^{\mathrm{ns}}(\PP^1,a-n;2n)$ is smooth of dimension $ \dim \widetilde{\mathcal{G}}_{2a-1,a}^{\mathrm{ns}}-2n.$
Define
$$q_n: \widetilde{\mathcal{M}}_{2a-1-n}^{\mathrm{ns}}(\PP^1,a-n;2n) \to \widetilde{\mathcal{G}}_{2a-1,a}^{\mathrm{ns}}$$
by sending a map $f:C\to \PP^1$ with markings $x_1, \ldots, x_n, y_1, \ldots, y_n$ to the stable map $$\widetilde{f}:X \to \PP^1,$$
where $X:=C \cup R_1\cup \ldots \cup R_n$, with $R_i \cong \PP^1$, $R_i \cap C=\{x_i,y_i\}$, and with $\tilde{f}_{C}=f$ and $\deg(\tilde{f}_{R_i})=1$ for $i=1, \ldots, n$. We let $\widetilde{X}$ be the stabilization of $X$. Set
\begin{equation}\label{def-zn}
Z_n:=q_n^{-1}(\mbox{pr}_1(\mathfrak{B}^{\mathrm{ns}}_a)).
\end{equation}
Each component of $Z_n$ has dimension at least $\dim \widetilde{\mathcal{G}}_{2a-1,a}^{\mathrm{ns}}-2n-1$.

\vskip 3pt

The following theorem, giving a classification of points in $Z_n$, will be proved in the last section of the paper and is crucial in completing the proof of Schreyer's conjecture.

\begin{thm} \label{hard-thm}
Fix $a\geq 3$ and $0 \leq n \leq 2a-5$, and let $C$ be an integral nodal curve of genus $2a-1-n$  with a unique pencil $f: C \to \PP^1$ of degree $a-n$. Choose pairs  $(x_i,y_i)$ of
smooth distinct points of $C$ for $i=1, \ldots, n$. Assume $f(x_i) \neq f(y_i)$ and that for any subset $S \seq \{x_1,y_1, \ldots, x_n,y_n\}$ of cardinality at most $n$,
one has $h^0\bigl(C,f^*(\mathcal{O}_{\PP^1}(2))(\sum_{s \in S}s)\bigr) =3.$
Then if $\pi^{-1}\bigl([\widetilde{X}]\bigr)$ consists of a unique point, then $[f,x_1,y_1, \ldots, x_n,y_n]$ does not lie in $Z_n$.
\end{thm}
The hypothesis that $C$ has a unique pencil of degree $a-n$ amounts to $W^1_{a-n}(C)=\{A \}$, with $A$ being a base-point free line bundle with $h^0(A)=2$, which implies $C$ 
has gonality $a-n$. If $C$ is an integral, nodal curve of genus $2a-1-n$ having a  pencil $f$ of degree $a-n$ satisfying $h^0(C, f^*\mathcal{O}_{\PP^1}(2))=3$ and if 
$\{x_1,y_1, \ldots, x_n, y_n\}$ is a general set of points on $C$, then $$h^0\Bigl(C,f^*(\mathcal{O}_{\PP^1}(2))(\sum_{s \in S}s)\Bigr) =3$$ for a subset $S \seq \{x_1,y_1, \ldots, 
x_n, y_n\}$  with $|S|\leq n+1$. Indeed, $h^0\bigl(C,\omega_C\otimes f^*\mathcal{O}_{\PP^1}(-2)\bigr)=n+1$, and since $(x_i,y_i)$ are general, we find
$h^0\bigl(C, \omega_C\otimes f^*\mathcal{O}_{\PP^1}(-2)(-\sum_{s\in S}s)\bigr)=n+1-|S|$. Thus
$h^0\bigl(C, f^*\mathcal{O}_{\PP^1}(2)(\sum_{s \in S} s)\bigr)=3$.

\vskip 2pt

In particular,  the following is an immediate corollary of Theorem \ref{hard-thm}.
\begin{thm} \label{outside-BN-divisor}
Let $C$ be a curve of genus $g$ and gonality $k\leq \frac{g+1}{2}$ satisfying bpf-linear growth. Assume that there is a unique $A \in G^1_k(C)$ and that we have $h^0(C, A^{\otimes 2})=3$. Choose general pairs of points $(x_i,y_i)$ on $C$ for $i=1, \ldots, g-2k+1$ and
let $D$ be the nodal curve obtained by identifying  $x_i$ and $y_i$ for all $i$. Then $\pi^{-1}([D])\cap \overline{\mathfrak{BN}}=\emptyset$.
\end{thm}

The proof of Theorem \ref{hard-thm} (and thus that of Theorem \ref{outside-BN-divisor}) is surprisingly involved and takes up the last section of the paper. To avoid disrupting the logical flaw of the paper, we assume Theorem \ref{hard-thm} and proceed towards the proof of Theorems \ref{goneric0} and \ref{goneric}.

\section{Extending the Eagon-Northcott Divisor} \label{extending}

We begin by recalling the definition of the Eagon-Northcott from the Introduction.
\begin{definition}
The Eagon-Northcott divisor $\mathcal{EN} \seq \mathcal{H}_{2g-2k+1,g-k+1}$ is defined as the locus of covers $[f: C \to \PP^1]$ such that $\dim K_{g-k,1}(C,\omega_C)>g-k$.
\end{definition}

We extend $\mathcal{EN}$ as a determinantal locus over a partial compactification of  $\mathcal{H}_{2g-2k+1,g-k+1}$. From Theorem \ref{hirsch} and
\cite[Proposition 4.10]{SSW}, observe that we have the following equality of subsets of $\H_{2g-2k+1,g-k+1}$:
$$\mathfrak{BN}=\mathcal{EN}.$$

\vskip 3pt

We construct an extension of  $\mathcal{EN}$ on the moduli space $  \widetilde{\mathcal{G}}_{2a-1,a}^{\mathrm{ns}}:= \widetilde{\mathcal{M}}_{2a-1}^{\mathrm{ns}}(\PP^1,a)$
of stable maps from the previous section. Precisely, we construct the extended Eagon-Northcott divisor
$$\widetilde{\mathcal{EN}} \seq  \widetilde{\mathcal{G}}_{2a-1,a}^{\mathrm{ns}}$$ by studying the minimal free resolutions of the scrolls attached to a cover
$[f: C \to \PP^1] \in \widetilde{\mathcal{G}}_{2a-1,a}^{\mathrm{ns}}$. Set $A:=f^*(\OO_{\PP^1}(1))\in W^1_{a}(C)$. Since $f$ is finite and flat, $f_*\OO_C$ is locally free
and we write $f_* \mathcal{O}_C\cong \OO_{\PP^1}\oplus \mathcal{E}_f^{\vee}$, where $\mathcal{E}_f$ is the so-called \emph{Tschirnhausen
bundle} of $f$, admitting a splitting
$$\mathcal{E}_f=\OO_{\PP^1}(e_1)\oplus \cdots \oplus \OO_{\PP^1}(e_{a-1}),$$
where $e_1\leq \ldots \leq e_{a-1}$ are the \emph{scrollar invariants} of $f$ and satisfy $e_1+\cdots+e_{a-1}=3a-2$.  Dualising the morphism $\mathcal{O}_{\PP^1} \to f_* \mathcal{O}_C$ leads
to an exact sequence
$$0 \longrightarrow \mathcal{E}_f \longrightarrow f_*\omega_f \longrightarrow \mathcal{O}_{\PP^1} \longrightarrow 0.$$

We tensor the morphism $f^*(\mathcal{E}_f) \to \omega_f$ by $f^*\omega_{\PP^1}$ and  produce a morphism
$\displaystyle f^*(\mathcal{E}_f(-2)) \twoheadrightarrow \omega_C$, inducing a closed immersion, see \cite{schreyer1}, or \cite{casnati-ekedahl}
$$j: C \to  \PP\bigl(\mathcal{E}_f(-2)\bigr).$$

Note that  $\mathcal{E}_f(-2)$ is a globally generated vector bundle on $\PP^1$ with  $\deg (\mathcal{E}_f(-2))=a$. Denoting by
$\varphi:X:=\PP\bigl(\mathcal{E}_f(-2)\bigr)\rightarrow \PP^1$ the
associated $(a-1)$-dimensional scroll, we  have a morphism
$$\iota:X \to \PP\bigl(H^0(\PP^1,\mathcal{E}_f(-2))\bigr)\cong \PP^{2a-2},$$
such that $\displaystyle \iota \circ j: C \to \PP^{2a-2}$ is the canonical morphism of $C$, cf.\ \cite{schreyer1}. Observe that since $C$ has no disconnected nodes, $\omega_C$
is globally generated. Also observe
that if $h^0(C,A^{\otimes 2})=3$, then $e_1\geq 3$ and  $\iota$ is a closed immersion.

\vskip 4pt

The  Picard group of the scroll $X$ is generated by the class of a ruling $R:=\varphi^{*}(\mathcal{O}_{\PP^1}(1))$ together with $H:=\OO_X(1)$. Note that
$H^0(X,H)\cong H^0(C,\omega_C)$, whereas $H^0(X,R)\cong H^0(C,A)$ and  $H^0(X,
\OO_X(H-R))\cong H^0(C,\omega_C\otimes A^{\vee})$. As already mentioned in the Introduction, the Eagon-Northcott complex,
explicitly describes the minimal free resolution of
$$\Gamma_X(H):=\bigoplus_{q \in \mathbb{Z}}H^0(X, H^{\otimes q}),$$
as a $\text{Sym}\ H^0(X,H)$-module, see \cite{schreyer1}. This gives that
$$K_{p,0}(X, H)=0 \mbox{ for  } p>0, \mbox{  whereas  } \ K_{p,q}(X, H)=0,  \mbox{ for  } q\geq 2 \mbox{ and any } p,$$
as well as the canonical identifications
$$
K_{p,1}(X,H)\cong  \bigwedge^{p+1} H^0(X,H-R)\otimes \mbox{Sym}^{p-1} H^0(X,R) \otimes \bigwedge^2 H^0(X, R)$$
$$\cong \bigwedge^{p+1} H^0(C,\omega_C\otimes A^{\vee})\otimes
\mbox{Sym}^{p-1} H^0(C,A) \otimes \bigwedge^2 H^0(C,A).\\
$$
In particular, $b_{p,1}(X,H)=p{a\choose p+1}$.

\vskip 4pt

We record the following lemma, while skipping the proof:
\begin{lem} \label{coh-van}
We have the vanishing $H^i(X,H^{\otimes q})=0$, for $i \geq 1$ and  $q \geq 0$. Furthermore, $H^i(X,\mathcal{O}_{X}(-H))=0$, for $i \geq 2$.
\end{lem}

Define the kernel bundles $M_{H}$ and $M_{\omega_C}$ on $X$ and $C$ respectively by the exact sequences
\begin{align*}
&0 \longrightarrow M_H \longrightarrow H^0(X,H)\otimes \mathcal{O}_{X} \longrightarrow H \longrightarrow  0 \\
&0 \longrightarrow M_{\omega_C} \longrightarrow H^0(C, \omega_C) \otimes \mathcal{O}_C \longrightarrow \omega_C \longrightarrow 0.
\end{align*}
As $C\seq X$ is linearly normal, $j^*M_H \cong M_{\omega_C}$. Note that $H^0\bigl(X, \bigwedge^p M_H\bigr)=H^0\bigl(C, \bigwedge^p M_{\omega_C}\bigr)=0$, for $p\geq 1$.
Further, we record the following short exact
sequences:
\begin{align}
& 0 \longrightarrow \bigwedge^{p+1} M_H \otimes \OO_X\bigl((q-1)H\bigr) \longrightarrow  \bigwedge^{p+1} H^0(X,H) \otimes \OO_X\bigl((q-1)\bigr) \longrightarrow
\bigwedge^{p} M_H \otimes \OO_X\bigl(qH\bigr)
\longrightarrow 0, \label{ses-kernel-H}\\
& 0 \longrightarrow \bigwedge^{p+1} M_{\omega_C} \otimes \omega_C^{\otimes (q-1)} \longrightarrow \bigwedge^{p+1} H^0(C,\omega_C) \otimes \omega_C^{\otimes (q-1)} \longrightarrow
\bigwedge^{p} M_{\omega_C} \otimes \omega_C^{\otimes q} \longrightarrow 0 \label{ses-kernel-C}.
\end{align}

We shall make use of the following vanishing statement.
\begin{lem} \label{coh-van-kb}
We have $H^i\bigl(X,\bigwedge^p M_H \otimes H^{\otimes q})\bigr)=0$ for $i \geq 2$ and arbitrary $p, q \geq 0$.
\end{lem}
\begin{proof}
By the sequence (\ref{ses-kernel-H}) and Lemma \ref{coh-van}, it suffices to show $H^{i-1}\bigl(\bigwedge^{p-1}M_H\otimes H^{\otimes (q+1)})\bigr)=0$. Continuing in this
fashion, it suffices to show $H^1\bigl(X, \bigwedge^{p-i+1}M_H\otimes H^{\otimes (q+i-1)})\bigr)=0$. Since $H^1\bigl(X,H^{\otimes (q+i-1)}\bigr)=0$, this amounts to $K_{p,q+i}(X, H)=0$,
which holds as $q+i \geq 2$.
\end{proof}

\begin{lem} \label{betaf}
There is an injective restriction map of linear syzygies
$$\alpha_f: K_{a-1,1}(X,H)\rightarrow K_{a-1,1}(C, \omega_C).$$
The map $\alpha_f$ is surjective if and only if the restriction map
$$ \beta_f : \; H^0\Bigl(X, \bigwedge^{a-2}M_H\otimes H^{\otimes 2}\Bigr) \to H^0\Bigl(C, \bigwedge^{a-2}M_{\omega_C} \otimes \omega_C^{\otimes 2}\Bigr)$$
is injective.
\end{lem}
\begin{proof}
The map $\alpha_f$ fits into a commutative diagram with exact rows:
$$
\xymatrix{
0 \ar[r]  &\bigwedge^{a} H^0(X,H)  \ar[r] \ar[d]^{\cong} &H^0\bigl(X,\bigwedge^{a-1}M_H \otimes H\bigr)
\ar[r]  \ar[d]^{\mathrm{res}_C} &K_{a-1,1}(X,H) \ar[r] \ar[d]^{\alpha_f} &0 \\
0 \ar[r] &\bigwedge^{a} H^0(C,\omega_C)  \ar[r] &H^0\bigl(C,\bigwedge^{a-1}M_{\omega_C} \otimes \omega_C\bigr) \ar[r] &K_{a-1,1}(C, \omega_C) \ar[r] &0
}.$$
Since $C\seq X$ is linearly normal, it follows that $\mbox{res}_C$ is injective, therefore $\alpha_f$ is injective as well. On the other hand, by the snake lemma
the surjectivity of $\alpha_f$ is equivalent to the surjectivity of $\mathrm{res}_C$.
From the kernel bundle description of Koszul cohomology, we write
$$K_{a-2,2}(X,H)=\mbox{Ker}\Bigl\{H^1\bigl(X,\bigwedge^{a-1} M_H\otimes H\bigr)\rightarrow \bigwedge^{a-1}H^0(X,H)\otimes H^1(X,H)\Bigr\}.$$
Since $H^1(X,H)=0$ and $K_{a-2,2}(X,H)=0$, it follows $H^1\bigl(X,\bigwedge^{a-1}M_H \otimes H\bigr)=0$. We write the following diagram with exact rows:
{\small{$$\xymatrix@C=1em{
0 \ar[r]  & H^0\bigl(X,\bigwedge^{a-1}M_H \otimes H) \ar[r] \ar[d]^{\mbox{res}_C} & \bigwedge^{a-1} H^0(X,H) \otimes H^0(X,H)
\ar[r]  \ar[d]^{\cong} &H^0\bigl(X, \bigwedge^{a-2}M_H \otimes H^{\otimes 2}\bigr) \ar[r] \ar[d]^{\beta_f} &0 \\
0 \ar[r]  & H^0\bigl(C, \bigwedge^{a-1}M_{\omega_C} \otimes \omega_C\bigr) \ar[r] & \bigwedge^{a-1} H^0(C,\omega_C) \otimes H^0(C,\omega_C)
\ar[r] &H^0\bigl(C, \bigwedge^{a-2}M_{\omega_C} \otimes \omega_C^{\otimes 2}\bigr)
},$$}}

By the snake lemma, the surjectivity of $\mbox{res}_C$ is equivalent to the injectivity of $\beta_f$.
\end{proof}

Koszul duality gives an isomorphism $K_{a-2,2}(C,\omega_C)\cong K_{a-1,1}(C,\omega_C)^{\vee}$, therefore we have a surjection
$$H^0\bigl(C,\bigwedge^{a-2} M_{\omega_C}\otimes \omega_C^{\otimes 2}\bigr)\longrightarrow H^0\bigl(C,\bigwedge^{a-2} M_{\omega_C}\otimes \omega_C^{\otimes 2}\bigr)/
\bigwedge^{a-1} H^0(C,\omega_C)\otimes H^0(C,\omega_C) \cong K_{a-1,1}(C,\omega_C)^{\vee}.$$
The composition of this map with $\alpha_f^{\vee}$ gives rise to a surjection
$$\psi_f:H^0\bigl(C,\bigwedge^{a-2}M_{\omega_C}\otimes \omega_C^{\otimes 2}\bigr)\rightarrow K_{a-1,1}(X,H)^{\vee}.$$
Because $K_{a-2,2}(X,H)=0$, from the second diagram in the proof of Lemma \ref{betaf}, it follows $\psi_f\circ \beta_f=0$.

\begin{lem} \label{ident-kern}
We have a natural isomorphism $\mathrm{Ker}(\psi_f) \cong H^2\bigl(X, \bigwedge^{a}M_H \otimes I_{C/X}\bigr)^{\vee}$.
\end{lem}
\begin{proof}
Since $H^1(X,\OO_X)=0$, the description of Koszul cohomology via kernel bundles yields the identification $K_{a-1,1}(X,H)^{\vee} \cong H^1\bigl(X, \bigwedge^{a}{M_H}\bigr)^{\vee}$.
Using that $\bigwedge^{a-2}M_{\omega_C}\otimes \omega_C\cong \bigwedge^a
M_{\omega_C}^{\vee}$, Serre-Duality gives the isomorphism
\begin{align*}
H^0\bigl(C, \bigwedge^{a-2}M_{\omega_C} \otimes \omega_C^{\otimes 2}\bigr)^{\vee} \cong H^1\bigl(C, \bigwedge^{a }M_{\omega_C}\bigr),
\end{align*}
which enables us to identify the dual map $\psi_f^{\vee}$ with the restriction
$$H^1\bigl(X,\bigwedge^{a}{M_H}\bigr) \rightarrow H^1\bigl(C, \bigwedge^{a}M_{\omega_C}\bigr).$$
Then $\mbox{Ker}(\psi_f)\cong \mbox{Coker}\bigl(\psi_f^{\vee})\bigr)^{\vee} \cong H^2\bigl(X, \bigwedge^{a}M_H \otimes I_{C/X}\bigr)^{\vee}$, using also Lemma \ref{coh-van-kb}.
\end{proof}

\vskip 4pt

Putting the above pieces together, we have constructed a natural map
$$\beta_f  : \;  H^0\bigl(X,\bigwedge^{a-2}M_H \otimes H^{\otimes 2}\bigr) \to H^2\bigl(X,\bigwedge^{a}M_H \otimes I_{C/X}\bigr)^{\vee}$$
such that  $b_{a-1,1}(C, \omega_C) > a-1$ if and only if $\beta_f$ fails to be injective. We shall see that both sides of this map have the same dimension.
This allows us to construct $\widetilde{\mathcal{EN}}$ as the degeneracy locus of a morphism between vector bundles of the same rank on the space of stable maps.

\begin{lem} \label{constant-dim-for-grauert}
We have:
\begin{align*}h^0\bigl(X,\bigwedge^{a-2}M_H \otimes H^{\otimes 2}\bigr)&=h^2\bigl(X, \bigwedge^{a}M_H \otimes I_{C/X}\bigr)=(2a-2) {2a-1 \choose a}-a+1. \\
\end{align*}
\end{lem}
\begin{proof}
As already pointed out  $H^1\bigl(X,\bigwedge^{a-1}M_H \otimes H\bigr)=0$. Therefore
$$ h^0\bigl(X,\bigwedge^{a-2}M_H \otimes H^{\otimes 2}\bigr)=(2a-1) {2a-1 \choose a}-h^0\bigl(X,\bigwedge^{a-1}M_H \otimes H\bigr),$$
by the short exact sequence (\ref{ses-kernel-H}). We further have a short exact sequence
$$0 \longrightarrow \bigwedge^{a}H^0(X,H) \longrightarrow H^0\bigl(X,\bigwedge^{a-1}M_H \otimes H\bigr) \longrightarrow K_{a-1,1} (X,H) \longrightarrow 0,$$
thus using that $b_{a-1,1}(X,H)=a-1$, we find $h^0\bigl(X,\bigwedge^{a-1}M_H \otimes H\bigr)=a-1+{2a-1 \choose a}$, which leads to the claimed formula for $h^0\bigl(X,\bigwedge^{a-2}M_H\otimes H^{\otimes 2}\bigr)$.

\vskip 3pt

Using Lemma \ref{ident-kern}, we compute:
$$h^2\bigl(X,\bigwedge^{a}M_H\otimes I_{C/X}\bigr)=\mbox{dim}(\mbox{Ker } \psi_f)=h^0\bigl(C,\bigwedge^{a-2}M_{\omega_C}\otimes \omega_C^{\otimes 2}\bigr)-b_{a-1,1}(X,H).$$
Recall that $b_{a-1,1}(X,H)=a-1$.  The Riemann-Roch theorem (still valid for a nodal curve $C$ with no disconnecting nodes)
gives $$ h^0\bigl(C,\bigwedge^{a-2}M_{\omega_C}\otimes \omega_C^{\otimes 2}\bigr)=\chi\bigl(C,\bigwedge^{a-2} M_{\omega_C} \otimes \omega_C^{\otimes 2}\bigr)=(4a-2){2a-2 \choose a},$$
which finishes the proof.
\end{proof}

\vskip 4pt

We now explain how the above considerations can be carried out in a relative setting. Let $$\xymatrix{
\mathcal{C} \ar[r]^-f \ar[rd]_{\nu} &\mathcal{P} \ar[d]^{\mu}\\
& \widetilde{\mathcal{G}}_{2a-1,a}^{\mathrm{ns}}
}$$
be the universal degree $a$ cover, where $\mathcal{P} = \widetilde{\mathcal{G}}_{2a-1,a}^{ns} \times \PP^1$.
The universal Tschirnhausen bundle $\mathcal{E}_f$ on $\mathcal{P}$ fits into an exact sequence:
$$ 0 \longrightarrow \mathcal{E}_{f} \longrightarrow f_{*} \omega_{f} \longrightarrow \mathcal{O}_{\mathcal{P}} \longrightarrow 0.$$

We further have the projective bundle
$\varphi: \; \mathcal{X}:=\PP(\mathcal{E}_f \otimes \omega_{\mu}) \to \mathcal{P}$
and a closed immersion
$j: \mathcal{C} \hookrightarrow \mathcal{X}.$ Set $h:= \mu\circ \varphi: \;  \PP(\mathcal{E}_f \otimes \omega_{\mu}) \to \widetilde{\mathcal{G}}_{2a-1,a}^{\mathrm{ns}}.$
By Grauert's Theorem, $h_{*}\bigl(\OO_{\mathcal{X}}(1)\bigr)$ is a vector bundle of rank $2a-1$.
Define the determinant $ \xi:=\det h_{*}\bigl(\mathcal{O}_{\mathcal{X}}(1)\bigr).$
The evaluation map $h^*h_*\OO_{\mathcal{X}}(1)\rightarrow \OO_{\mathcal{X}}(1)$ is furthermore surjective, thus we can define the kernel bundle $\mathcal{M}$ by
$$0 \longrightarrow \mathcal{M} \longrightarrow h^*h_*(\OO_{\mathcal{X}}(1)\rightarrow \mathcal{O}_{\mathcal{X}}(1)\longrightarrow 0.$$ Then $\mathcal{M}$ restricts to the kernel bundle
$M_H$ for each scroll induced by an element $[C\rightarrow \PP^1]$.
Note that $j$ is defined by the surjection
$f^*(\mathcal{E}_f \otimes \omega_{\mu}) \rightarrow \omega_f \otimes f^*\omega_{\mu} \cong \omega_{\nu} ,$
hence  $\mathcal{O}_{{\mathcal{C}}}(1) \cong \omega_{\nu}$.
Set $$ \mathcal{F}_1:=h_*\Bigl(\bigwedge^{a-2}\mathcal{M} \otimes \OO_{\mathcal{X}}(2)\Bigr) \otimes \xi^{\vee},$$
which is a vector bundle of rank $(2a-2) {2a-1 \choose a}-a+1$, by Lemma \ref{constant-dim-for-grauert}.
Set $$ \mathcal{F}_2:=h_*\Bigl(\bigwedge^{a-2}\mathcal{M} \otimes \OO_{\mathcal{C}}(2)\Bigr) \otimes \xi^{\vee},$$ which is a vector bundle of rank
$(2a-2) {2a-1 \choose a}$. Restriction to $\mathcal{C}$ induces a morphism
$$\beta: \; \mathcal{F}_1 \to \mathcal{F}_2. $$
Relative duality gives the isomorphism
$$R^1\nu_*\bigl(\bigwedge^a \cM_{|\mathcal{C}}\bigr) \cong \Bigl(\nu_*\bigl(\bigwedge^{a}\mathcal{M}^{\vee}_{|\mathcal{C}}\otimes \omega_{\nu}\bigr)\Bigr)^{\vee} \cong \mathcal{F}_2^{\vee},
$$
using $\det (\mathcal{M})\cong h^*\xi\otimes \OO_{\X}(-1)$. Define the rank $a-1$ vector bundle by
$\mathcal{F}_3:=R^1h_*\Bigl(\bigwedge^{a}\mathcal{M}\Bigr)^{\vee},$
The dual of the restriction morphism $\psi^{\vee }: \; R^1h_*\bigl(\bigwedge^{a}\mathcal{M}\bigr) \to R^1\nu_*\bigl(\bigwedge^{a}\mathcal{M}_{|\mathcal{C}}\bigr)$  gives a morphism
$$\psi: \; \mathcal{F}_2 \to \mathcal{F}_3$$
with fibre over a moduli point $[f:C\rightarrow \PP^1]$ equal to $\psi_f$. As already explained, $\displaystyle \psi \circ \beta=0.$

\vskip 3pt

We get a short exact sequence of vector bundles over $\ph$:
$$0 \longrightarrow R^1h_*\bigl(\bigwedge^{a}\mathcal{M}\bigr) \longrightarrow R^1\nu_*\bigl(\bigwedge^{a}\mathcal{M}\otimes \OO_{\mathcal{C}}\bigr) \longrightarrow
R^2h_*\bigl(\bigwedge^{a}\mathcal{M}\otimes I_{\mathcal{C}/\X}\bigr)\longrightarrow 0, $$
where $\mathcal{F}_4:=R^2h_*\bigl(\bigwedge^{a}\mathcal{M} \otimes I_{\mathcal{C}/\X}\bigr)$ is a vector bundle of rank $(2a-2) {2a-1 \choose a}-a+1$ by Lemma \ref{constant-dim-for-grauert}.
Thus we may canonically identify
$$\text{Ker}(\psi) \cong \mathcal{F}^{\vee}_4$$ and we have an induced morphism between vector bundles
$ \beta: \mathcal{F}_1 \to \mathcal{F}^{\vee}_4$
globalizing the morphisms $\beta_f$ as the moduli point $[f]\in \ph$ varies. Since $\mbox{rk}(\mathcal{F}_1)=\mbox{rk}(\mathcal{F}_4)$, we define the extended Eagon-Northcott divisor
$$\widetilde{\mathcal{EN}} \seq  \ph$$ as the degeneracy locus of $\beta$. By the results of the previous chapter, this is a genuine divisor.

\vskip 4pt

Define $\mathcal{EN}^{\mathrm{sm}}$ as the union of all components of $\widetilde{\mathcal{EN}}$ containing an element $[f: C \to \PP^1]$, with $C$ being a smooth curve and all ramification simple.
The following lemma is a direct consequence of Theorem \ref{outside-BN-divisor}.
\begin{lem} \label{almost-done}
Let $C$ be a curve of genus $g$ and gonality $k\leq \frac{g+1}{2}$ satisfying bpf-linear growth. For $i=1, \ldots, g-2k+1$, choose pairs $(x_i,y_i)$ of general points on $C$ and let
$B$ be the semistable curve given as the union of $C$ with $g-2k+1$ smooth rational curves $R_i$  meeting the rest of $B$ precisely at $x_i, y_i$.
Let  $$[f: B \to \PP^1] \in \widetilde{\mathcal{G}}_{2g-2k+1,g-k+1}^{\mathrm{ns}}$$ be a morphism with $\mathrm{deg}(f_{C})=k$ and $f_{R_i}$ an isomorphism. Assume that $f_{C}$ is the unique minimal pencil on $C$ and, further, $h^0(C,f^*\mathcal{O}_{\PP^1}(2))=3$. Then $[f] \notin \mathcal{EN}^{\mathrm{sm}}$.
\end{lem}
\begin{proof}
Consider the closure $\overline{\mathcal{EN}}^{\mathrm{sm}} \seq \mm_{2g-2k+1}(\PP^1,g-k+1)$ in the moduli space of stable maps. We have the projections $\pi':\mm_{2g-2k+1}(\PP^1,g-k+1)\to
\overline{\mathcal{M}}_{2g-2k+1}$, as well as the projection $\pi$ from the space of admissible covers. There is an equality of closed sets
$\pi'(\overline{\mathcal{EN}}^{\mathrm{sm}})=\pi(\overline{\mathcal{EN}})$, since $\mm_{2g-2k+1}(\PP^1,g-k+1)$ is a $PGL(2)$-cover of $\hh_{2g-2k+1,g-k+1}$ over the open set of morphisms with
smooth source and simple ramification. By Theorem \ref{outside-BN-divisor}, the point $[D] \in \overline{\mathcal{M}}_{2g-2k+1}$ defined by the stabilization of $B$ does not lie in
$\pi'(\overline{\mathcal{EN}}^{\mathrm{sm}})$, therefore, \ $[f] \notin \overline{\mathcal{EN}}^{\mathrm{sm}}$.
\end{proof}

To complete the proof of Theorem \ref{goneric} (assuming Theorem \ref{hard-thm}), we need to show that, in the situation of Lemma \ref{almost-done}, the point $[f]$ does not lie in the extended Eagon-Northcott divisor
$\widetilde{\mathcal{EN}}$. Note that $[f]$ lies in precisely one boundary divisor of $\widetilde{\mathcal{G}}_{2g-2k+1,g-k+1}^{\mathrm{ns}}$, namely the divisor $\Delta$ whose general point corresponds to maps $h:C\rightarrow \PP^1$, where $C$ is a union of two curves $C_1$ and $C_2$ of genera $g-1$ and $0$ respectively, meeting at two points, and such that $\mbox{deg}(h_{C_1})=g-k$ and $\mbox{deg}(h_{C_2})=1$.
Since $\widetilde{\mathcal{EN}}$ is pure of codimension one,  we need to show that $\widetilde{\mathcal{EN}}$ does not contain $\Delta$. We carry this out in the next section, using $K3$ surfaces.

\section{K3 Surfaces and Schreyer's Conjecture} \label{k3}

We start by considering a $K3$ surface $X=X_d$ with Picard group
generated by two classes $L$ and $E$ with self intersections given by $(L)^2=4d-4$, $(E)^2=0$ and $(L \cdot E)=d$, for $d \geq 3$.
By performing Picard-Lefschetz transformations and a reflection if necessary, we may assume that  $L$ is big and nef.

\begin{lem}
For $X$ as above, the class $L$ is base point free and $E$ is the class of a smooth elliptic curve.
\end{lem}
\begin{proof}
We firstly show that $L$ is base point free. As $L$ is big and nef, it suffices to show there is no smooth elliptic curve $F$ with $(L \cdot F) =1$, see \cite[Proposition 8]{mayer}.
Assume such $F$ exists, and
write $F=aL+bE$ for $a,b \in \mathbb{Z}$. As $F$ is smooth and elliptic, $(F)^2=0$, implying $0=(aL+bE) \cdot F=a+b(E \cdot F)=a(1+db)$. If $a=0$, then $(L \cdot F)=bd \neq 1$,
since $d \geq 2$, so $db=-1$, which is again impossible. Thus $L$ is base point free.

We next show that $E$ is the class of a smooth elliptic curve. As $(E)^2=0$ and $E$ is primitive, it suffices to show that $E$ is nef. Since $(E \cdot L)>0$, and $L$ is big and nef,
$E$ is effective.  Suppose $E$ is not nef.  Then there exists a smooth, rational curve $R$ with $(R \cdot E)<0$. Write $R=aL+bE$ for $a,b \in \mathbb{Z}$.
Then $(R \cdot E)<0$ implies $a<0$. As $(R)^2=-2$ and $R$ is effective, we must have $b>0$. We have $-2=(R)^2=R \cdot (aL+bE)=a(R \cdot L)+b(R \cdot E)=a\bigl((R \cdot L)+bd\bigr)$, which is
impossible for $d \geq 3$.
\end{proof}

We now discuss the Brill-Noether theory of a smooth curve $C \in |L|$. To that end, we follow  \cite[\S 2]{kemeny-singular} which works in the situation of a higher rank
Picard lattice containing the lattice $\text{Pic}(X_d)$.
\begin{lem} \label{lem2.5}
Let $D \in \mathrm{Pic}(X_d)$ be effective with $(D)^2 \geq 0$. Assume in addition $L-D$ is effective and $(L-D)^2 > 0$. Then $D=cE$, for some integer $c$.
\end{lem}
\begin{proof}
This is a slight modification of \cite[Lemma 2.5]{kemeny-singular}.
Write $D=aL+bE$. As $L-D$ is effective and $E$ nef, $(L-D) \cdot E=(1-a)(L\cdot E) \geq 0$, so $a \leq 1$. From $(D \cdot E) \geq 0$, we obtain $a \geq 0$. If $a=1$, then
$(L-D)^2=b^2(E)^2=0$, so we must have $a=0$ as required.
\end{proof}
The next lemma describes the Brill-Noether behaviour of  curves in the linear system $|L|$.
\begin{lem}\label{bnbeh}
Let $C \in |L|$ be a smooth curve. Then $\mathrm{Cliff}(C)=d-2$ and  $W^1_d(C)$ is reduced and consists of the single point
$\mathcal{O}_C(E)$.
\end{lem}
\begin{proof}
The proof that $\mbox{Cliff}(C)=d-2$ is essentially the same as \cite[Lemma 2.6]{kemeny-singular}. Arguing as in \cite[Lemmas 2.7, 2.8]{kemeny-singular},
we see that $W^1_d(C)$ is set-theoretically a single point, namely $\OO_C(E)$.

It remains to establish that $W^1_d(C)$ is reduced, which amounts to showing that  $h^0(\mathcal{O}_C(2E))=3$.
From the exact sequence
$$0 \longrightarrow \OO_X(E) \longrightarrow \OO_X(2E) \longrightarrow  \OO_E(2E)\cong \mathcal{O}_E \longrightarrow 0,$$
we deduce $h^1(X, 2E)=1$ and then $h^0(X,2E)=3$ by Riemann--Roch.
By the exact sequence
$$0 \longrightarrow \OO_X(2E-C) \longrightarrow \OO_X(2E) \longrightarrow \OO_C(2E) \longrightarrow 0,$$
it suffices to show $h^0(X, 2E-C)=h^1(X, 2E-C)=0$. As $(C-2E)^2=-4$, by Riemann-Roch, it suffices to show that neither $2E-C$ nor $C-2E$ are effective.
As $(E \cdot 2E-C)<0$ and $E$ is nef, $2E-C$ is not effective. Now suppose $C-2E$ is effective with integral components $R_1, \ldots, R_{\ell}$, for $\ell \geq 1$.
We write $R_i=a_iL+b_iE$, for integers $a_i, b_i$, with $\sum_{i=1}^{\ell} a_i=1$ and $\sum_{i=1}^{\ell} b_i=-2$. As $(E \cdot R_i) \geq 0$, we find $a_i \geq 0$ for all $i$.
Without loss of generality, we may assume $a_1=1$ and $a_i=0$ for $2\leq i\leq \ell$. As $R_i$ is integral, we must then have $b_i=1$ for $i>1$. Thus $R_1=L-(\ell+1)E$, which implies
$(R_1)^2=4d-4-2d(\ell+1) \leq -4$, contradicting that $R_1$ is integral.
\end{proof}
We can now prove Theorem \ref{goneric}, that is, establish the Schreyer Conjecture.
\begin{proof}[Proof of Theorem \ref{goneric}]
Let $[f: B \to \PP^1]$ be as in the statement of Lemma \ref{almost-done}. By an argument along the lines of  \cite[Corollary 1]{V1}, we have an injection
$K_{g-k,1}(C, \omega_C) \hookrightarrow K_{g-k,1}(B, \omega_B)$. For the sake of completeness we recall the proof.

\vskip 3pt

The Mayer-Vietoris sequence induces an injection $H^0(C,\omega_C) \hookrightarrow H^0(B,\omega_B)$, as well as the composition of injections
$H^0(C,\omega_C^{\otimes 2}) \hookrightarrow
H^0\bigl(C,\omega_C^{\otimes 2}(\sum_{i=1}^{g-2k+1}(x_i+y_i))\bigr)\hookrightarrow H^0(B,\omega_B^{\otimes 2})$. We then get a commutative diagram
$$\xymatrix{
\bigwedge^{g-k+1} H^0(\omega_C) \ar[r]^{\delta_0 \; \; \;} \ar[d] &\bigwedge^{g-k}H^0(\omega_C)\otimes H^0(\omega_C)
\ar[r]  \ar[d] &\bigwedge^{g-k-1}H^0(\omega_C) \otimes H^0(\omega_C^{\otimes 2})  \ar[d] \\
\bigwedge^{g-k+1} H^0(\omega_B)  \ar[r]^{\delta'_0 \; \; \;} &\bigwedge^{g-k}H^0(\omega_B)\otimes H^0(\omega_B) \ar[r] &\bigwedge^{g-k-1}H^0(\omega_C) \otimes H^0(\omega_C^{\otimes 2})
}.$$
The conclusion now follows from the existence of  maps, see also \cite[Lemma 7.1]{aprodu-nagel}
$$\wedge: \; \wedge^{g-k}H^0(\omega_C)\otimes H^0(\omega_C) \to \wedge^{g-k+1} H^0(\omega_C)  \ \mbox{ and }
\wedge': \wedge^{g-k}H^0(\omega_B)\otimes H^0(\omega_B) \to \wedge^{g-k+1} H^0(\omega_B), $$
with $\wedge \circ \delta_0= \pm (g-k)\mbox{Id}$ and $\wedge' \circ \delta'_0= \pm (g-k)\mbox{Id}$.

\vskip 4pt

We secondly claim that $[f]$ does not lie in the extended Koszul divisor $\widetilde{\mathcal{EN}}$. In light of the injective map above, this will complete the proof.
As $[f]$ lies in exactly one boundary divisor, namely $\Delta$, all that remains is to show that the divisor $\widetilde{\mathcal{EN}}$ does not contain $\Delta$.
By  Lemma \ref{bnbeh}, we know that any smooth curve $C \in |L|$ on the K3 surface $X=X_{g-k+1}$ satisfies
$b_{g-k,1}(C,\omega_C)=g-k$. By the Lefschetz Theorem for Koszul cohomology  \cite{green-koszul}, the same holds for any integral nodal curve $C_0 \in |L|$. As any integral, nodal curve $C_0$
(with at least one node) defines a point in $\Delta$, it suffices to show that such curves exist for the general $X_{g-k+1}$.

\vskip 3pt

In order to do this, it suffices to take $2g-2k+1 \geq 8$, as the conclusion of the Theorem is well-known for $g \leq 8$ by \cite{schreyer1}. Indeed, if $g-k \leq 3$, then, since we are assuming $g \geq 2k-1$, we must have $k \leq 4$ and $g \leq 7$. The class $L-E$ is very ample for a general $K3$ surface $X_{g-k+1}$ general with the given Picard lattice, by degenerating to the $K3$ surface $Y_{\Omega_{2g-2k+1}}$ from
\cite[Lemma 2.3]{kemeny-singular}.
Choose a curve $C_1 \in |L-E|$ meeting a smooth elliptic curve $E_0 \in |E|$ transversally, and consider the nodal curve $C_1 \cup E_0$. Pick any node
$p_1 \in C_1 \cup E_0$. Then, by \cite[Theorem \ 3.8]{tannenbaum}, the moduli space $\bar{\mathcal{V}}_1(X_d)$ parametrising deformation of $C_1 \cup E_0$ preserving the assigned node $p_1$
is smooth near $(C_1 \cup E_0, p_1)$ of dimension $2g-2k$. As $\dim |L-E|+\dim |E|=g-k+1<g-1$ for $k \geq 3$, there exist integral $1$-nodal curves $C_0 \in |L|$,
completing the proof.
\end{proof}

\section{Twistings and pencils on singular curves}

In this section we prove Theorem \ref{hard-thm}. We need to construct twistings of line bundles on reducible curves. Suppose we have a family $\mathcal{B} \to \Delta$ of nodal curves over a smooth base $\Delta$, with smooth general fibre and with special fibre $C \cup R$, where $R \cong \PP^1$ and $C$ is smooth meeting $R$ transversally in two points. If the total family $\mathcal{B}$ is smooth, then $\mathcal{O}_{\mathcal{B}}(R)$ is a line bundle which is trivial outside of the special fibre. We will generalize the definition of this twist to the case where the general fibre is only integral and $\mathcal{B}$ is not necessarily smooth, so $R$ may not be a Cartier divisor on $\mathcal{B}$.

We first introduce some convenient notation.
\begin{definition} \label{conv-notation}
Let $X$ be a connected, nodal curve and $p \in X$. The ``blow-up" $c_p: \widetilde{X} \to X$ of $X$ at $p$ is defined as such:
\begin{enumerate}
\item If $p \in X$ is a node, let $\nu: X' \to X$ denote the partial normalisation of $X$ at $p$ and let $\nu^{-1}(p)=\{a,b\}$. Then $\widetilde{X}$ is defined to be $X' \cup E$, where $E \cong \PP^1$ and $E \cap X'=\{ a,b\}$.
\item If $p$ is not a node, then define $\widetilde{X}=X \cup E$, where $E \cong \PP^1$ with $X \cap E=\{p\}$. We further define the ``strict transform" $X'$ of $X$ to be the closure $\overline{\widetilde{X}\setminus E}$.
\end{enumerate}
In both cases, $c_p: \widetilde{X} \to X$ is given by contracting the unstable component $E$ to the point $p$.
\end{definition}

An \emph{abelian differential} of type $(1)^{2g-2}$ on a smooth curve $C$ of genus $g$ is an ordered marking $\alpha$ of degree $2g-2$ such that there exists a nonzero section $s \in H^0(C,\omega_C)$ with $s(\alpha)=0$. Such $(C,\alpha)$ define elements of $\mm_{g,2g-2}$.  Let $\mm_{g}\bigl((1)^{2g-2}\bigr) \seq \mm_{g,2g-2}$ be the space of twisted abelian differentials of type $(1)^{2g-2}$, i.e.\ the closure of the space of abelian differentials
of this type, \cite{far-pand}.

The first twisting construction we will use is described below and is rather well-known. We attach a proof due to lack of a suitable reference.
\begin{prop} \label{twisting-prop}
Let $(\Delta,0)$ be an irreducible, pointed, variety and
$\mathcal{B} \to \Delta$
be a flat family of nodal curves of genus $g \geq 2$ such that the fibre $B_{t}$ is irreducible for a general $t \in \Delta$ and
$$B_0 \cong C \cup R, \; \; R \cong \PP^1, \; \; R \cap C=\{u,v\},$$
and $C$ is irreducible. After a base change, there is a birational morphism $\nu: \widetilde{\mathcal{B}} \to \mathcal{B}$ of families of nodal curves over $\Delta$ and a line bundle $\tau \in \mathrm{Pic}(\widetilde{\mathcal{B}})$, such that one of the following cases occur:
\begin{enumerate}
\item $\nu$ is an isomorphism (and hence $\widetilde{B}_0 \cong B_0$). Furthermore $\tau_C \cong \mathcal{O}_C(u+v)$ and $\deg({\tau_R})=-2$.
\item $\widetilde{B}_0$ is a blow-up of $B_0$ at a node $p \in \{u,v\}$ with exceptional component $E$. Identifying $R$ and $C$ with their strict transforms, ${\tau_C} \cong \mathcal{O}_C(u+v)$ and  $\deg(\tau_R)=\deg(\tau_E)=-1$.
\end{enumerate}
\end{prop}
Informally, case $(1)$ corresponds to twisting by a line bundle, whereas case $(2)$ corresponds to twisting by a torsion-free sheaf.
\begin{proof}
Performing a base change if necessary, we choose markings $p_i: \Delta \to  \mathcal{B}$, for  $i=1, \ldots, g-1$, with $p_{g-1}(0) \in R \seq B_0 $ and  $p_{1}(0), \ldots, p_{g-2}(0)$ being general points of $C$. For all $t \in \Delta$, up to scaling, there exists a unique form $0\neq s_t \in H^0(B_t, \omega_{B_t})$  vanishing along $p_1(t)+\cdots +p_{g-1}(t)$. Note that $s_0$ vanishes identically on $R$. By the generality of the points $p_{i}(0)$ for $i=1,\ldots, g-2$, the restriction $s_{0|C}$ vanishes on an abelian differential of type $(1)^{2g-4}$. Hence, after possibly shrinking $\Delta$ and a further finite base change, those components of the vanishing set of $s_t$ which limit to points on $C$ provide additional sections $p_i: \Delta \to  \mathcal{B}$,  for  $i=g, \ldots, 2g-3,$
with $\sum _{i \neq g-1} p_i(0)$ defining an abelian differential on $C$, such that $[B_t, p_1(t), \ldots, p_{2g-3}(t)]$ lies in the image of the morphism
$\mm_{g}\bigl((1)^{2g-2}\bigr) \to \mm_{g,2g-3}$ obtained by forgetting the last marking. We obtain a map $\nu: \widetilde{\mathcal{B}} \to \mathcal{B}$ between families of nodal curves,
together with sections $q_i : \Delta \to \widetilde{\mathcal{B}}$ for    $i=1, \ldots, 2g-2,$
such that $[\widetilde{B}_t, q_1(t), \ldots, q_{2g-2}(t)] \in \mm_{g}\bigl((1)^{2g-2}\bigr)$ and, further, after forgetting the last marking $q_{2g-2}(t)$, these  curves stabilize to $[B_t,
p_1(t), \ldots, p_{2g-3}(t)]$.

\vskip 3pt

Notice that $\nu\bigl(q_{2g-2}(0)\bigr) \in R$. We now distinguish three cases. If $\nu\bigl(q_{2g-2}(0)\bigr) \notin \{u, v, p_{g-1}(0) \}$, then $\nu$ is an isomorphism, and we may take
$$\tau  :=\omega_{\mathcal{B}/ \Delta}\Bigl(-\sum_{i=1}^{2g-2} (\nu\circ q_i)(\Delta) \Bigr).$$
If $\nu(q_{2g-2}(0))=p_{g-1}(0)$,  then $\nu\circ q_{2g-2}$ defines a section in the smooth locus of each fibre $B_t$, and we define $\tau$ by the same formula. Lastly, suppose
$\nu(q_{2g-2}(0)) \in R \cap C$. We first observe that $\nu(q_{2g-2}(t))$ is not a node of $B_t$ for a general $t$. Indeed, otherwise $\widetilde{B}_t$ would be a blow-up of $B_t$ at a node. Further, there must be a nonzero section of $\omega_{\widetilde{B}_t}$ vanishing at the exceptional component $E$, as well as on smooth points $q_1(t), \ldots, q_{2g-3}(t)$, which  is impossible. Hence, in the last case, $\widetilde{\mathcal{B}}$ is a blow-up of $\mathcal{B}$ at a node in the intersection $C \cap R$, and we may now take $\tau  :=\omega_{\widetilde{\mathcal{B}}/ \Delta}\Bigl(-\sum_{i=1}^{2g-2} q_i(\Delta) \Bigr)$.
\end{proof}

The next twisting result is more sophisticated. Under further hypotheses, it ensures that we can put ourselves in the more degenerate case $(2)$ of Proposition \ref{twisting-prop}.

\begin{prop} \label{construction-ind-step}
Keeping the notation of Proposition \ref{twisting-prop}, assume additionally  we have a family
$$f: \mathcal{B} \to \PP^1\times \Delta$$
of stable maps of degree $k$ with $\mathrm{deg}(f_{0|R})=1$, \ $h^0(B_0, f^*_0\mathcal{O}_{\PP^1}(1))=2$,
and  $\omega_C \otimes f^*_{C} \mathcal{O}_{\PP^1}(-1)$  base point free. Assume the locus of points $t \in \Delta$ for which $B_t$ is reducible has codimension two. Then, after a base change, we have $\nu: \widetilde{\mathcal{B}} \to \mathcal{B}$ and $\tau \in \mathrm{Pic}(\widetilde{\mathcal{B}})$ as in case $(2)$ of Proposition \ref{twisting-prop}.
\end{prop}
\begin{proof}
Let $\mm^{\dagger} \seq \mm_{g}(\PP^1,k;2g-2-k)$ denote the substack of the moduli space of degree $k$ stable maps $h: B \to \PP^1$ with markings $p_1, \ldots, p_{2g-2-k}$ defined by the following conditions (i) $h^0(B,h^*\mathcal{O}_{\PP^1}(1))=2$ and  (ii) $H^0\bigl(B,\omega_B \otimes h^*\mathcal{O}_{\PP^1}(-1)(-p_1-\cdots-p_{2g-2-k})\bigr)\neq 0$.

The moduli space $\mm^{\dagger}$ may be constructed from an incidence variety in the obvious fashion, see also  \cite[\S 2]{BCGGM}.
We have a forgetful morphism $\mm^{\dagger} \to \mm_{g}(\PP^1,k)$. Let $$r: \Delta^{\dagger}:=\Delta \times_{\mm_{g}(\PP^1,k)} \mm^{\dagger} \to \Delta,$$
where $\Delta \to \mm_{g}(\PP^1,k)$ is induced by the family of stable maps $f:\mathcal{B} \to \PP^1\times \Delta$.

\vskip 3pt

We denote by $\widetilde{B}_0$ the blow-up of $B_0=C\cup_{\{u,v\}} R$ at $u$ and set
$\widetilde{f}_0:=f_0\circ c_u: \widetilde{B}_0 \to \PP^1$. Choose distinct points $p_1, \ldots, p_{2g-3-k}$ in the smooth locus of $B_0$, with
$\sum_{i=1}^{2g-3-k} p_i \in |\omega_C \otimes f_{C}^*\mathcal{O}_{\PP^1}(-1)|$ and pick a general point $p_{2g-2-k} \in E$. Then $$t:=[\widetilde{f}_0, p_1, \ldots, p_{2g-2-k}] \in \Delta^{\dagger}.$$

We claim that any component $I \seq \Delta^{\dagger}$ containing the point $t$ dominates $\Delta$ under the forgetful morphism $r$. We are then done, by replacing $\Delta$ with $I$, setting $\widetilde{f}:=f\circ \nu:\widetilde{\mathcal{B}} \to \PP^1\times I$ to be the universal family and choosing
$$\tau:=\widetilde{f}^*\mathcal{O}_{\PP^1\times I}(-1) \otimes \omega_{\widetilde{\mathcal{B}}/I}\Bigl(-\sum_{i=1}^{2g-2-k}\widetilde{p}_i\Bigr),$$
where $ \widetilde{p}_i: I \to  \widetilde{\mathcal{B}} $ are the markings.

\vskip 3pt

From the construction of $\Delta^{\dagger}$ as an incidence variety it follows that
 $$\dim I \geq \dim \Delta+\dim |\omega_{B_0} \otimes f_0^*\mathcal{O}_{\PP^1}(-1)|. $$
 Since $\dim|\omega_C \otimes f_C^* \mathcal{O}_{\PP^1}(-1)|=\dim |\omega_{B_0} \otimes f_0^*\mathcal{O}_{\PP^1}(-1)|$, it follows that $r(I)$ has codimension at most one in $\Delta$. By assumption, the general point of $r(I)$ corresponds to a stable map with \emph{irreducible} source $B$. But, in an open subset about $t \in I$, the fibre of $r$ over such a stable map with irreducible base has expected dimension $\dim |\omega_{B_0} \otimes f_0^*\mathcal{O}_{\PP^1}(-1)|$. Thus $I$ dominates $\Delta$.
\end{proof}

In fact, by the exact same argument, we could equally well ensure that we are in $(1)$ of Proposition \ref{twisting-prop}, under the assumptions of Proposition \ref{construction-ind-step}.
\begin{prop} \label{construction-ind-step-II}
Keeping the notation of Proposition \ref{twisting-prop}, assume additionally  we have a family
$$f: \mathcal{B} \to \PP^1\times \Delta$$
of stable maps of degree $k$ with $\mathrm{deg}(f_{0|R})=1$, \ $h^0(B_0, f^*_0\mathcal{O}_{\PP^1}(1))=2$,
and  $\omega_C \otimes f^*_{C} \mathcal{O}_{\PP^1}(-1)$  base point free. Assume the locus of points $t \in \Delta$ for which $B_t$ is reducible has codimension two. Then, after a base change, we have $\tau \in \mathrm{Pic}(\mathcal{B})$ as in case $(1)$ of Proposition \ref{twisting-prop}.
\end{prop}
\begin{proof}
Identical to the proof of Proposition \ref{construction-ind-step}, except that we choose $p_{2g-2-k} \in R \seq B_0$.
\end{proof}

It is worth remarking that the construction of the twisting bundle $\tau$ in Propositions \ref{construction-ind-step} and \ref{construction-ind-step-II} commutes with base change, in the obvious sense.

\subsection{Induction Step}
Our task is now to prove the induction step of Theorem \ref{hard-thm}.
We first prove a weakening of the induction step, using the more basic Proposition \ref{twisting-prop}.
\begin{prop} \label{1-node-special-case}
Assume Theorem \ref{hard-thm} holds for $n=m$. Let $C$ be a general curve of genus $2a-2-m$ and gonality $a-m-1$ with pencil $f: C \to \PP^1$ of degree $a-m-1$. Let $(x_i,y_i)$ be general pairs of points in $C$ for $i=1, \ldots, m+1$. Then $[f, x_1,y_1, \ldots, x_{m+1},y_{m+1}]\notin Z_{m+1}$. In particular, $Z_{m+1}$ has codimension one in $\widetilde{\mathcal{M}}_{2a-1-(m+1)}^{\mathrm{ns}}\bigl(\PP^1,a-(m+1);2(m+1)\bigr)$.
\end{prop}

The conclusion of Proposition \ref{1-node-special-case}  make it possible to later apply  Proposition  \ref{construction-ind-step} and thus finish the proof of Theorem \ref{hard-thm}.
\begin{proof}
Suppose $[f,x_1,y_1, \ldots, x_{m+1},y_{m+1}] \in Z_{m+1}$. Let $B_{m+1}:=C \cup R_{m+1}$ with $R_{m+1} \cong \PP^1$ and $R_{m+1} \cap C=\{x_{m+1},y_{m+1}\}$ and let
$f_{m+1}: B_{m+1} \to \PP^1$ be the stable map with $f_{{m+1}|C}=f$ and  $\deg(f_{{m+1}|R_{m+1}})=1$. Our hypothesis implies
$$t=[f_{m+1},x_1,y_1, \ldots, x_m,y_m] \in Z_{m}.$$
Let $J \seq Z_m$ be any irreducible component containing $t$. As $\dim J \geq \dim \widetilde{\mathcal{G}}_{2a-1,a}^{\mathrm{ns}}-2m-1$ but $\dim Z_{m+1} \leq \dim \widetilde{\mathcal{M}}_{2a-2-m}^{\mathrm{ns}}(\PP^1,a-m-1;2(m+1))=\dim \widetilde{\mathcal{G}}_{2a-1,a}^{\mathrm{ns}}-2m-2$, the general point
$$[g: T\to \PP^1, x'_1,y'_1, \ldots, x'_m,y'_m]$$ of $J$ is a marked stable map with \emph{irreducible} source curve $T$. Let $\widetilde{B}_{m+1}$ be the curve obtained from $B_{m+1}$ by glueing $x_i$ to $y_i$ for $i=1, \ldots, m$ and contracting the unstable component $R_{m+1}$. Then $\widetilde{B}_{m+1}$ has gonality $a$ by Proposition \ref{tf-limit}. It follows that the curve $\widetilde{T}$ obtained from $T$ by glueing $x'_i$ to $y'_i$ for $i=1, \ldots, m$ has gonality at least $a$, in particular $\text{gon}(T)=a-m$.

\vskip 4pt

Let $S \seq  \{x'_1,y'_1, \ldots, x_m',y_m'\}$ be a set of cardinality at most $m$. As explained, a Riemann--Roch calculation gives $h^0\bigl(C,f^*\mathcal{O}_{\PP^1}(2)(\sum_{s \in S'} s)\bigr)=3$ for any $S' \seq \{x_1,y_1, \ldots, x_{m+1}, y_{m+1}\}$ of cardinality $|S'| \leq m+2$. Further, there is a unique admissible cover in $\pi^{-1}([\widetilde{B}_{m+1}])$, by Proposition  \ref{adm-limit}. In Lemma \ref{lemma-first-part-thm} and Proposition \ref{uniqueness-AssumpI}, which follow, we establish the following statements:
\begin{itemize}
\item $h^0\bigl(T,g^*(\mathcal{O}_{\PP^1}(2))(\sum_{s \in S}s)\bigr) =3.$
\item $\pi^{-1}([\widetilde{T}])$ consists of a unique admissible cover.
\end{itemize}
All these facts yield a contradiction by the assumption that Theorem \ref{hard-thm} holds for $n=m$.
\end{proof}

For the next two lemmas we fix the following notation. We let $C$ be an integral, nodal curve of genus $2a-2-m$ and gonality $a-m-1$, together with a pencil $f: C \to \PP^1$ of degree $a-m-1$. For $i=1,\ldots, m+1$, let $x_i,y_i\in C_{\mathrm{reg}}$ be distinct points such that $f(x_i) \neq f(y_i)$. Let $B_{m+1}:=C \cup R_{m+1}$ with $R_{m+1} \cong \PP^1$ and $R_{m+1} \cap C=\{x_{m+1},y_{m+1}\}$ and we denote by
\begin{equation}\label{fm}
f_{m+1}: B_{m+1} \to \PP^1
\end{equation}
the map with $f_{{m+1}|C}=f$ and $\deg_{R_{m+1}}(f_{m+1})=1$. Finally,  $\widetilde{B}_{m+1}$ is the curve obtained from $B_{m+1}$ by glueing $x_i$ to $y_i$ for $i=1, \ldots, m$ and contracting the unstable component $R_{m+1}$.

\begin{lem} \label{lemma-first-part-thm}
Let $(\Delta, 0)$ be an irreducible, pointed variety and let $$\phi: \mathcal{B} \to \PP^1\times \Delta, \; \;  \sigma_i,\tau_i: \; \Delta \to \mathcal{B} \mbox{ for } \ i=1, \ldots, m$$ be a family of marked stable maps over $(\Delta, 0)$ such that $\phi_{0}=f_{m+1}:B_{m+1}\rightarrow \PP^1$, $\sigma_i(0)=x_i$ and  $\tau_i(0)=y_i$ for $i=1, \ldots, m$.
For a general point $t\in \Delta$, assume the fibre $B_{t}$ is irreducible and let $S_t \seq \{\sigma_1(t), \tau_1(t), \ldots, \sigma_m(t), \tau_m(t)\}$ be a set of cardinality at most $m$.
\vskip 3pt
\renewcommand{\theenumi}{\Roman{enumi}}%
 \begin{enumerate}
 \item If $h^0\bigl(C,f^*\mathcal{O}_{\PP^1}(2)(\sum_{s \in S'} s)\bigr)=3$ for any subset $S' \seq \{x_1,y_1, \ldots, x_{m+1}, y_{m+1}\}$ of cardinality at most $m+2$, then $h^0\bigl(B_t,\phi_t^*\mathcal{O}_{\PP^1}(2)(\sum_{s_t \in S_t}s_t)\bigr) =3.$
 \vskip 3pt
 \item Assume  $h^0\bigl(C,f^*\mathcal{O}_{\PP^1}(2)(\sum_{s \in S'} s)\bigr)=3$ for any subset $S' \seq \{x_1,y_1, \ldots, x_{m+1}, y_{m+1}\}$ of cardinality at most $m+1$. If  $B_{t}$ is irreducible for $t \in \Delta$ outside a set of codimension two and  $\omega_C \otimes f^*\mathcal{O}_{\PP^1}(-1)$ is base point free, then $h^0\bigl(B_t,\phi_t^*\mathcal{O}_{\PP^1}(2)(\sum_{s_t \in S_t}s_t)\bigr) =3.$
 \end{enumerate}
\end{lem}
\begin{proof}
We choose a set of sections $\mathcal{S} \seq \{\sigma_1,\tau_1, \ldots, \sigma_m, \tau_m\}$ of cardinality at most $m$ and set
$$\mathcal{M}:=\phi^*\mathcal{O}_{\PP^1\times \Delta}(2)\bigl(\sum_{s \in \mathcal{S}} s\bigr)\in \mathrm{Pic}(\mathcal{B}).$$
We shall prove that $h^0(B_{t}, \mathcal{M}_t)=3$, for a general $t \in \Delta$. There exists a birational map $\nu: \widetilde{\mathcal{B}} \to \mathcal{B}$
together with a line bundle $\tau\in \mbox{Pic}(\widetilde{\mathcal{B}})$ enjoying the properties listed in Proposition \ref{twisting-prop}. Furthermore, when Assumption (II) applies,
then, by Proposition \ref{construction-ind-step}, we may ensure that $\tau$ is as in case $(2)$ of Proposition \ref{twisting-prop}.
Set $\mathcal{L}:=\nu^* \mathcal{M} \otimes \tau\in \mbox{Pic}(\widetilde{\mathcal{B}})$. Since $h^0(B_t,\phi_t^*\OO_{\PP^1}(2))\geq 3$, it suffices to show
$h^0(\widetilde{B},\mathcal{L}_0)\leq 3$,
where $\widetilde{B}$ denotes the fibre of $\widetilde{\mathcal{B}}$ over the point $0\in \Delta$.

\vskip 3pt

Assume first we are in case $(1)$ of Proposition \ref{twisting-prop}, thus $\widetilde{B}=B_{m+1}$. Then $C \seq \widetilde{B}$  and
$\mathcal{L}_{C} \cong f^*\mathcal{O}_{\PP^1}(2)\bigl(\sum_{s \in S} s+x_{m+1}+y_{m+1}\bigr)$,
where $S=\{s(0):s \in \mathcal{S}\}\subseteq \{x_1, y_1, \ldots, x_m,y_m\}$. Furthermore, $\mathcal{L}_{R_{m+1}}\cong \mathcal{O}_{R_{m+1}}$ and the claim follows by twisting the exact sequence
$$0 \longrightarrow \mathcal{O}_{R_{m+1}}(-2) \longrightarrow  \mathcal{O}_{B_{m+1}} \longrightarrow \mathcal{O}_C \longrightarrow 0$$
by $\mathcal{L}$ and using our assumptions.

\vskip 3pt

Assume now $\widetilde{B}$ is as in case $(2)$  of Proposition  \ref{twisting-prop}. Then $$\mbox(\deg \mathcal{L}_{R_{m+1}},  \deg \mathcal{L}_{E}\bigr)=(1,-1) \mbox{ and }
\ \;\mathcal{L}_C \cong f^*\mathcal{O}_{\PP^1}(2)\bigl(\sum_{s \in S} s+x_{m+1}+y_{m+1}\bigr).$$ Set $B':=C \cup R_{m+1} \seq \widetilde{B}$. Via
the exact sequence $0 \to \mathcal{O}_{B'}(-z-w) \to \mathcal{O}_{\widetilde{B}} \to \mathcal{O}_E \to 0$,
where $\{w\}=E \cap R_{m+1}$ and $\{z\}=C \cap E$, it suffices to show $h^0(B',\mathcal{L}_{B'}(-z-w))=3$,
following from $$h^0\bigl(C, f^*\mathcal{O}_{\PP^1}(2)(\sum_{s \in S} s+z')\bigr)=3,$$
where $\{z'\}=C \cap R_{m+1} \subseteq \{x_{m+1}, y_{m+1} \}$, which is a consequence of our assumptions.
\end{proof}

\vskip 3pt

\noindent {\bf{Set-up for the rest of the paper.}} We fix  and irreducible, pointed variety $(\Delta, 0)$  and let 
$\bigl(\phi: \mathcal{B} \to \PP^1\times \Delta, \; \;  \sigma_i,\tau_i: \; \Delta \to \mathcal{B}\bigr)$ be a family of marked stable maps of degree $a-m$ such that 
$\phi_{0}=f_{m+1}:B_{m+1}\rightarrow \PP^1$, $\sigma_i(0)=x_i$ and  $\tau_i(0)=y_i$ for $i=1, \ldots, m$. Assume the general fibre $B_{t}$ is an irreducible curve of genus $2a-m-1$. 
For $t\in \Delta\setminus \{0\}$, let $X_t$ be the curve obtained from the fibre $B_{t}$ by identifying the points $\sigma_i(t)$ and $\tau_i(t)$ for $i=1, \ldots, m$. 
As $t \to 0$, this tends to the stable curve $X_0:= \widetilde{B}_{m+1}$ of genus $2a-1$.

\medskip

With this set-up, we analyse the potential admissible covers in $\pi^{-1}([X_t])$ for general $t$. To begin with, we always have at least one element in the set $\pi^{-1}([X_t])$. Namely, 
if $\mu_t: B_t \to X_t$ denotes the partial normalization for $t \neq 0$, we consider the torsion-free sheaf
$$(\mu_t)_*\bigl(\phi_t^* \mathcal{O}_{\PP^1}(1)\bigr) \in W^1_a(X_t).$$
By the construction of \cite[Theorem 5]{ha-mu}, we may associate admissible covers corresponding to these torsion-free sheaves. We call such covers associated via this construction the \emph{Harris-Mumford covers} over $[X_t]$.

For $t=0$, let $\mu: C \to \widetilde{B}_{m+1}$ denote the partial normalization. Then we associate a Harris--Mumford cover $[f': B' \to T]\in \pi^{-1}([\widetilde{B}_{m+1}])$ to the torsion free sheaf $\mu_* f^*\mathcal{O}_{\PP^1}(1) \in W^1_a(\widetilde{B}_{m+1})$. It will prove useful to describe the Harris--Mumford cover $[f': B' \to T]$ over $X_0=\widetilde{B}_{m+1}$ in more detail. Firstly, $B'$ has a component isomorphic to the normalization $\widetilde{C}$ of $C$. All other components are isomorphic to $\PP^1$. The restriction $f'_{\widetilde{C}}$ is the composition $\widetilde{f}=f\circ \nu$, of $f$ with the normalization morphism $\nu:\widetilde{C}\rightarrow C$.  The fact that $\pi^{-1}([\widetilde{B}_{m+1}])$ consists of a single element implies that the pencil $f: C \to \PP^1$ is  unramified at the points $x_i, y_i$ for $i=1,\ldots, m+1$. Moreover, $\widetilde{f}$ is unramified at all points above nodes of $C$. Set $\PP^1_0:=f'(\widetilde{C})$. For each pair $(x_i,y_i)$ the curve $B'$ contains a rational chain with three components which meets $\widetilde{C}$ in $x_i,y_i$, with precisely one component dominating $\PP^1_0$. The map $f'$ has degree one on this component dominating $\PP^1_0$ and degree two on the other two components (cf.\ Proposition \ref{adm-limit}). For each node $p \in C$ with points $q_1, q_2 \in \widetilde{C}$ over $p$, there is a single rational component $R_p \cong\PP^1$ of $B'$ which meets $\widetilde{C}$ at $q_1, q_2$, and which does \emph{not} map to $\PP^1_0$. The admissible cover $f'$ has degree two on this component $R_p$. There are further rational tails attached to $\widetilde{C}$, none of which dominate $\PP^1_0$.

Let $\bigl(\phi: \mathcal{B} \to \PP^1\times \Delta, \; \;  \sigma_i,\tau_i \bigr)$ be as above and assume $\pi^{-1}([\widetilde{B}_{m+1}])$ consists of a single point, which must necessarily be a Harris--Mumford cover. Under certain assumptions, we shall prove that $\pi^{-1}([X_t])$ contains at least two elements for a general $t\in \Delta$. We firstly show how to construct certain stable maps from extraneous elements of $\pi^{-1}([X_t])$.

\begin{prop} \label{lem-unique}
Let $\bigl(\phi: \mathcal{B} \to \PP^1\times \Delta, \; \;  \sigma_i,\tau_i \bigr)$ be as above. Assume $\pi^{-1}([\widetilde{B}_{m+1}])$ consists of a single point and suppose $\pi^{-1}([X_t])$ contains at least two elements for a general $t\in\Delta$. Then, after replacing $\Delta$ by a finite base change, we have a second family $$\epsilon: \mathcal{Y}  \to \PP^1\times \Delta$$ of finite stable maps, with the following properties. Firstly, restricting to $\Delta^*:=\Delta \setminus \{0\}$, either $\mathcal{Y}^*:=\mathcal{Y} \times_{\Delta} \Delta^*$ is isomorphic to $\mathcal{B}^*:=\mathcal{B} \times_{\Delta} \Delta^*$, or there is a birational morphism $\mu_{q}: \mathcal{Y}^* \to \mathcal{B}^*$ such that the fibres of $\mu_{q}$ are partial normalizations $\mu_{q_t} : Y_t \to B_t$ at a single node $q_t \in B_t$. Secondly, the special fibre $Y_0$ is one of the following:
\begin{enumerate}
\item If $\mathcal{Y}^* \cong \mathcal{B}^*$, then the  fibre $Y_0$ is $B_{m+1}=C \cup R_{m+1}$ blown-up at some subset $S'$ of  $\{x_1,y_1, \ldots, x_m,y_m\}$ of cardinality at most $m$,
\item If $\mathcal{Y}^*\ncong \mathcal{B}^*$, then $Y_0$ is the blow-up of $C$ at some subset $S'$ of $\{x_1,y_1, \ldots, x_m,y_m\}$  of cardinality at most $m$,
\end{enumerate}
in the notation of Definition \ref{conv-notation}. Moreover, in all cases ${\epsilon_{0}}_{|C}=f$, $\deg_{E_z} (\epsilon_0)=1$ for any exceptional component $E_z$ at $z \in S'$ and lastly $\deg_{R_{m+1}}(\epsilon_0)=1$ in the case $R_{m+1} \seq Y_0$.

\medskip

Furthermore, in the case $\mathcal{Y}^* \cong \mathcal{B}^*$, $S'=\emptyset$, and $R_{m+1} \seq Y_0$, then the stable map $\sigma \circ \epsilon_t$ does not coincide with $\phi_t$ for general $t$ and any $\sigma \in \text{PGL}(2)$. Likewise, if $Y_t$ is the partial normalization at $q_t \in Y_t$, then $\sigma \circ \epsilon_t$ does not coincide with $\phi_t \circ \mu_{q_t}$ for any $\sigma \in \text{PGL}(2)$.
\end{prop}
\begin{proof}
By assumption, for general $t\in \Delta$ there exists an admissible cover $[f'_t: B'_t \to T_t]$ over $X_t$ which, (i) specializes to the unique element of $\pi^{-1}([\widetilde{B}_{m+1}])$ as $t\to 0$
and (ii) it is  distinct from the unique Harris--Mumford cover constructed from $(\mu_t)_*\bigl(\phi_t^* \mathcal{O}_{\PP^1}(1)\bigr)$.

\vskip 3pt

For general $t\in \Delta$, there exists a subcurve $B^n_{t} \seq B'_t$ isomorphic to the normalization of $B_{t}$. The limit of $B^n_{t}$ at $t\to 0$ is a connected subcurve $B^n_{0}$ of $B'$ containing $\widetilde{C}$. For each node $p \in C$, the component $R_p$ of $B'$ meeting $\widetilde{C}$ at $q_1, q_2$ either lies in  $B^n_{0}$, or else, there is a
smooth rational component $R_{p_t}$ of $B'_t$ meeting $B^n_{t}$ at two distinct points $q_{1,t}$ and $q_{2,t}$ and with $f'_t(q_{1,t})=f'_t(q_{2,t})$. As $t\to 0$, the curve  $R_{p,t}$ approaches $R_p$ and $q_{i,t}$ approaches $q_i$  for $i=1,2$.

\vskip 3pt

Let $\overline{B}^n_{t}$ be the union of $B^n_{t}$ with all components $R_{p,t}$ as above and let $D_t$ denote its stabilization. Define the finite stable map
$\chi_{t}: D_t \to \PP^1$ by first restricting $f'_t$ to $\overline{B}^n_{t}$, then contracting the target $T_t$ to $\PP^1_0:=f'_t(B^n_{t})$ and then lastly stabilizing the resulting morphism (this contracts each component $R_{p,t}$). As $t\to 0$, the map $\chi_{t}$ tends to the finite stable map
$$\chi_{0}: D_0 \to \PP^1,$$
where $D_0$ is a connected curve which is either:
\begin{enumerate}
\item The blow-up of $C$ at some set $S \seq \{x_1,y_1, \ldots, x_{m+1},y_{m+1}\}$  of cardinality at most $m+1$ containing at most one term from each pair $(x_i,y_i)$ for all $i$, or
\item  $B_{m+1}=C \cup_{\{x_{m+1}, y_{m+1}\}} R_{m+1}$ blown-up at some set $S \seq \{x_1,y_1, \ldots, x_m,y_m\}$ of cardinality at most $m$.
\end{enumerate}

Note that $\mbox{deg}(\chi_0)\leq a$. The two cases are distinguished depending on whether the curve $B_t$ limiting to $B_{m+1}$ smooths out the nodes $x_{m+1}$ and $y_{m+1}$. In the first case, the irreducible curve $B_t$ has a node $q_t$ and $D_t$ is isomorphic to the partial normalization $\mu_{q_t}: D_t \to B_{t}$ at $q_t$, with points $\sigma_{m+1}(t)$ and $\tau_{m+1}(t)$ lying over this node. In the second case, $D_t \cong B_{t}$. In both cases we define $C_2$ to be the curve formed by contracting all exceptional components $E_i$ of $D_0$.
If $E_1, \ldots, E_{\ell}$ are the exceptional components of $D_0$, where $\ell\leq m+1$, then $\deg_{E_i}(\chi_0)=1$ for $i=1, \ldots,\ell$ and  $\chi_{0|C}=f$. If, further, $R_{m+1} \seq D_0$ then $\deg_{R_{m+1}}(\chi_0)=1$.

\vskip 3pt

In the case $C_2=B_{m+1}$, that is,  when $D_0$ contains $R_{m+1}$, we must have $D_t=B_{t}$. In this situation, we set $Y_t:=D_t$, set $S'=S$ and define the finite stable map $\epsilon_{t}:=\chi_t:Y_t\rightarrow \PP^1$, for all $t \in \Delta$. Notice that the stable map $\epsilon_{t}$ cannot lie in the orbit of $\phi_{t}$ under the natural $PGL(2)$-action, or else the admissible cover $f'_t$ would coincide with the admissible cover constructed from $(\mu_{t})_*(\phi^*_{t} \mathcal{O}_{\PP^1}(1))$.

In the case $C_2=C$, that is, when $D_0$ does not contain $R_{m+1}$, there exist distinct points $\sigma_{m+1}(t)$ and $\tau_{m+1}(t)$ in the smooth locus of $D_t$, together with a connected chain of rational curves in $B'_t$ which meets $B^n_{t}$ at $\sigma_{m+1}(t)$ and $\tau_{m+1}(t)$. There are now two cases which we must consider separately.

First of all, assume $x_{m+1}, y_{m+1} \notin S$. Then, $(\sigma_{m+1}(t), \tau_{m+1}(t)) \to (x_{m+1},y_{m+1})$ as $t\to 0$. Since $f$ does not identify the points $x_{m+1}, y_{m+1} \in C$, the map $\chi_{t}: D_t \to \PP^1$ also cannot identify $\sigma_{m+1}(t)$ and  $\tau_{m+1}(t)$. Once again, we define $\epsilon_{t}:=\chi_t:Y_t\rightarrow \PP^1$ with $Y_t:=D_t$ and set $S'=S$. The stable map $\epsilon_{t}$ cannot lie in the orbit of $\phi_{t} \circ \mu_{q_t}$ under the natural $PGL(2)$-action, as the latter map identifies $\sigma_{m+1}(t)$ and  $\tau_{m+1}(t)$.

\vskip 3pt

In the second case $C_2=C$ and $x_{m+1} \in S$ or $y_{m+1} \in S$. In this case, $\chi_{t}$ identifies $\sigma_{m+1}(t)$ with $\tau_{m+1}(t)$. We set  $Y_t:=B_t$ and define $\epsilon_{t}:Y_{t} \to \PP^1$ for $t \neq 0$ as the unique map such that $\epsilon_{t} \circ \mu_{q_t}=\chi_{t}$. The limit of $\epsilon_{t}$ at $t\to 0$ is the stable map $\epsilon_{0}: \; Y_0 \to \PP^1,$ where $Y_{0}=D_0 \cup R_{m+1}$, $\deg_{R_{m+1}}(\epsilon_{0})=1$ and ${\epsilon_{0}}_{|_{D_0}}=\chi_0$. In this case $S':=S \setminus S \cap \{x_{m+1}, y_{m+1} \}$. Once again, $\epsilon_{t}$ cannot coincide with $\phi_{t}\circ \mu_{q_t}$ after automorphisms of the target, as otherwise $f'_t$ would coincide with the admissible cover constructed from $(\mu_{t})_*(\phi^*_{t} \mathcal{O}_{\PP^1}(1))$.

\vskip 3pt

After a base change, we have a family $\epsilon : \mathcal{Y} \to \PP^1 \times \Delta$ with fibres $\epsilon_t$ as above and with the claimed properties.
\end{proof}

\begin{eg}
We illustrate Proposition \ref{lem-unique} in the case $m=0$ with $C$ being a smooth curve of genus $2a-2$ and gonality $a-1$. Let $[f'_t: B'_t \to T_t]$ be an admissible cover over $X_t=B_t$ as in the proof of Proposition \ref{lem-unique}. There are two cases. If $B_t$ is smooth (and of genus $2a-1$), then in the notation of the proof, $D_t=B_t \seq B'_t$ and $\epsilon_t=\chi_t$ is a stable map $\epsilon_t: B_t \to \PP^1$ of the same degree as $\phi_t$ but not isomorphic to it, modulo $PGL(2)$.

In the second case, $B_t$ has a single node $q_t$. In this case, $B'_t$ contains a component $D_t=B^n_t$ given by the normalization of $B_t$, as well as a chain of rational components meeting $B^n_t$ at the points $\sigma_{1}(t), \tau_{1}(t)$ lying over $q_t$. Since $\lim_{t \to 0}f'_t=[f': B \to T] \in \pi^{-1}([\widetilde{B}_{1}])$, it follows from the explicit description of the admissible cover $f': B \to T$ that there are two possibilities. Firstly, the chain at $\sigma_{1}(t), \tau_{1}(t)$ may have three rational components, which happens when $f'_t(\sigma_1(t))\neq f'_t(\tau_1(t))$. In this case, $\lim_{t \to 0}D_t=D_0=C$ and $\{\sigma_{1}(t), \tau_{1}(t) \} \to \{ x_{1}, y_{1} \}$. We have that $\epsilon_t=\chi_t: D_t \to \PP^1$ is a stable map tending to $\epsilon_0=f$.

In the last case, $\chi_t: D_t \to \PP^1$ identifies $\sigma_{1}(t), \tau_{1}(t)$, that is, $f'_t(\sigma_1(t))=f'_t(\tau_1(t))$. In this case the chain at $\sigma_{1}(t), \tau_{1}(t)$ has a single component, say $\Gamma_t$. In this case, $\Gamma_t$ limits to a smooth, rational component of the base $B'$ of $f'$ which meets $C$ at $z \in \{ x_{1}, y_{1} \}$. Then $\chi_t: D_t \to \PP^1$ tends to $\chi_0: D_0 \to \PP^1$, with $D_0$ the blow-up of $C$ at $z'=\{ x_{1}, y_{1} \}\setminus \{z \}$. In this case, we set $\epsilon_t: B_t \to \PP^1$ to be the unique map factoring through $\chi_t$. Then $\epsilon_0$ has, as base, a curve isomorphic to $B_{1}=C \cup R_{1}$.
\end{eg}

The next two propositions are key technical results needed for the induction step for Theorem \ref{hard-thm}.
\begin{prop} \label{uniqueness-AssumpII}
Let $\bigl(\phi: \mathcal{B} \to \PP^1\times \Delta, \; \;  \sigma_i,\tau_i \bigr)$ be as above, with $\phi_0=f_{m+1}$ and with $B_t$ irreducible for $t \in \Delta$ outside a set of codimension two. Assume  $h^0\bigl(C,f^*\mathcal{O}_{\PP^1}(2)(\sum_{s \in S} s)\bigr)=3$ for any subset $S \seq \{x_1,y_1, \ldots, x_{m+1}, y_{m+1}\}$ of cardinality at most $m+1$. Assume further that $\omega_C \otimes f^*\mathcal{O}_{\PP^1}(-1)$ is base point free. Then if $\pi^{-1}([\widetilde{B}_{m+1}])$ consists of a single point, likewise $\pi^{-1}([X_t])$ consists of a unique element for general $t\in \Delta$.
\end{prop}
\begin{proof}
Suppose for a contradiction that $\pi^{-1}([X_t])$ contains at least two elements for general $t$. After a base change, there is a family $\epsilon: \mathcal{Y}  \to \PP^1\times \Delta$ of stable maps as in Proposition \ref{lem-unique}.

Assume firstly that $\mathcal{Y}^*$ is not isomorphic to $\mathcal{B}^*$. Replace $\Delta$ by a smooth curve mapping into $\Delta$ and passing through both $0$ and a sufficiently general $t$; in this way we may take the base $\Delta$ to be one-dimensional and smooth. Let $\mu': \mathcal{Y} \to \mathcal{Y}'$ be the morphism which contracts the unstable components $E_1, \ldots, E_{\ell}$ on the central fibre, $\ell=|S'| \leq m$, and is otherwise an isomorphism (this exists after a further base change). By Lemma \ref{lb-cont} below applied to $\epsilon^*\mathcal{O}_{\PP^1\times \Delta}(1)$, there exists a line bundle
$$\mathcal{N} \in \text{Pic}(\mathcal{Y}'),$$
with $\mathcal{N}_{t} \cong \epsilon^*_{t} \mathcal{O}_{\PP^1}(1)$ for $t \neq 0$, and
$\mathcal{N}_{0} \cong f^* \mathcal{O}_{\PP^1}(1) \bigl(\sum_{s \in S'}s\bigr)$. In this case, possibly after base change, $\mathcal{B}$ is singular along a section $q:\Delta\rightarrow \mathcal{B}$. Let $$\mu_q: \; \mathcal{B}' \to \mathcal{B}$$ be the partial normalization at $q$ and let $\mathcal{B}^{n}$ be the surface obtained from $\mathcal{B}'$ by contracting any exceptional components on the special fibre (there will be precisely one such component, over either $x_{m+1}$ or $y_{m+1}$). Possibly after a base change, we have $\mathcal{Y}' \cong \mathcal{B}^{n}$. Let
$$\phi':=\phi\circ \mu_q: \mathcal{B}'\to \PP^1\times \Delta.$$
Applying Lemma \ref{lb-cont}, there exists a line bundle $\mathcal{G} \in \mathrm{Pic}(\mathcal{B}^n) \cong \mathrm{Pic}(\mathcal{Y}')$ with
$\mathcal{G}_{t} \cong (\phi_t')^* \mathcal{O}_{\PP^1}(1)$ and with $\mathcal{G}_{0} \cong f^*\mathcal{O}_{\PP^1}(1)(z)$, where $z \in \{x_{m+1}, y_{m+1} \}.$ By the base-point free pencil trick $h^0(\mathcal{N}_{t} \otimes \mathcal{G}_{t}) \geq 4$ for general $t\in \Delta$. But our hypotheses give $h^0(\mathcal{N}_{0} \otimes \mathcal{G}_{0})=3$, which is a contradiction.

So we may assume $\mathcal{Y} \cong \mathcal{B}$. Thus, for a general $t\in \Delta$ we have a family of maps
$$(\phi_t, \epsilon_t) \; : B_t \to \PP^1 \times \PP^1.$$ Base-changing if necessary, we complete this to a family $j : \mathcal{J} \to \PP^1 \times \PP^1 \times \Delta$ of stable maps over $\Delta$, with $j_t=(\phi_t, \epsilon_t)$ for a general $t\in \Delta$ and such that $\phi$ (respectively  $\epsilon$) is obtained by stabilizing $\mbox{pr}_1 \circ j$ (respectively  $\mbox{pr}_2 \circ j$), where $\mbox{pr}_i : \PP^1 \times \PP^1 \to \PP^1$ are the projections for $i=1,2$. The central fibre $J_0$ contains exceptional components $E_1, \ldots, E_{\ell}$ over $\{z_1 \ldots, z_{\ell} \}=S'$, $\ell=|S'|$, and $j_0$ has degree one on these curves. We have a morphism $c: \mathcal{J} \to \mathcal{J}'$ contracting $E_1, \ldots E_{\ell}$, which factors through the stabilization morphism $\mathcal{J} \to \mathcal{B}$. Consider the line bundle $\mathcal{L} \in \text{Pic}(\mathcal{J})$ defined by $j^*\left(\mathcal{O}_{\PP^1}(1) \boxtimes \mathcal{O}_{\PP^1}(1) \right)$. Then $\deg_{E_i}(\mathcal{L})=1$ for all $i=1, \ldots, \ell$. We have $h^0(B_t,L_t) \geq 4$ for a general $t \in \Delta$.

\vskip 4pt

There are two cases. In the first case $\mathcal{J}' \cong \mathcal{B}$. Applying Proposition \ref{construction-ind-step}, we have $\nu: \widetilde{\mathcal{B}} \to \mathcal{B}$ and $\tau \in \mathrm{Pic}(\widetilde{\mathcal{B}})$ as in case $(2)$ of Proposition \ref{twisting-prop}. Choose a smooth curve $\Delta' $ with a map $i: \Delta' \to \Delta$ such that $i(\Delta')$ contains $0 \in \Delta$ as well as a sufficiently general $t \in \Delta$. We may now apply Lemma \ref{lb-cont} to the line bundle $\mathcal{L}$ to produce $\mathcal{L}' \in \text{Pic}(\mathcal{J}') \cong \text{Pic}(\mathcal{B})$.
In our case $$J'_0=B_{m+1}, \ {L'_0}_{|C} \cong f^*\mathcal{O}_{\PP^1}(2)( \sum_{s \in S'} s ), \; \; \; {L'_0}_{|{R_{m+1}}} \cong \mathcal{O}_{R_{m+1}}(2).$$
Consider $M:=\tau \otimes \nu^* \mathcal{L}'\in \mbox{Pic}(\widetilde{\mathcal{B}})$. By the same argument as Lemma \ref{lemma-first-part-thm}, $h^0(M_0) \leq 3$, which contradicts semicontinuity.

\vskip 4pt

In the second case $J'_0$ is the blow-up of $\mathcal{B}$ at some $z \in \{x_{m+1},y_{m+1} \}$. In this case, the exceptional component $E_z$ is contracted by $\phi$ but is identified with a non-contracted component of $\epsilon$. Applying Proposition \ref{construction-ind-step}, we have $\tau \in \mathrm{Pic}(\mathcal{B})$ as in case $(1)$ of Proposition \ref{twisting-prop}. Apply Lemma \ref{lb-cont} to the line bundle $\mathcal{L}$ to produce $\mathcal{L}' \in \text{Pic}(\mathcal{J}')$ and define $M:=s^*\tau \otimes \mathcal{L}'$,
where $s: \mathcal{J}' \to \mathcal{B}$ is the natural stabilization morphism. Thus $\deg_{E_z}(M_0)=1$ and $\deg_{R'_{m+1}}(M_0)=-1$, where $R'_{m+1} \seq J'_0$ is the strict transform of $R_{m+1}$. Once again $h^0(M_0) \leq 3$, contradicting semicontinuity.
\end{proof}
The following lemma was needed in the proof of Proposition \ref{uniqueness-AssumpII} above. The proof is well-known.
\begin{lem} \label{lb-cont}
Let $(\Delta,0)$ be a smooth pointed curve and $\mathcal{C} \to \Delta$ a proper family of nodal curves  central fibre $C_0=T \cup E$,
where $T$ is nodal and connected and $E \cong \PP^1$, with $\{z\}=E \cap  T$ a single point. Assume $z $ is an isolated singularity of the surface $\mathcal{C}$. Let $\mathcal{L}\in \mathrm{Pic}(\mathcal{C})$ be a line bundle with $\deg_E(\mathcal{L})=1$. If $\mu_E: \mathcal{C} \to \mathcal{C}'$ denotes the contraction morphism of $E$, then there exists a line bundle $\mathcal{L}'$ on $\mathcal{C}'$ with $\mathcal{L}'_t \cong \mathcal{L}_t$ for $t \neq 0$ and $\mathcal{L}'_0 \cong \mathcal{L}_{0|T}(z).$
\end{lem}

Note that after a finite base change, a contraction morphism as in the lemma above always exists. We now prove the analogue of Proposition \ref{uniqueness-AssumpII} under an amended hypothesis.

\begin{prop} \label{uniqueness-AssumpI}
Let $\bigl(\phi: \mathcal{B} \to \PP^1\times \Delta, \; \;  \sigma_i,\tau_i \bigr)$ be as above, with $B_t$ irreducible for general $t$ and $\phi_0=f_{m+1}$. 
Assume $|\pi^{-1}([\widetilde{B}_{m+1}])|=1$ and $h^0\bigl(C,f^*\mathcal{O}_{\PP^1}(2)(\sum_{s \in S} s)\bigr)=3$ for any subset $S \seq \{x_1,y_1, \ldots, x_{m+1}, y_{m+1}\}$ with $|S|\leq m+2$. Then $|\pi^{-1}([X_t])|=1$ for a general $t\in \Delta$.
\end{prop}
\begin{proof}
The fibre $\pi^{-1}([X_t])$ always contains at least one element, given by a Harris--Mumford cover. Suppose $\pi^{-1}([X_t])$ contains at least two elements for 
general $t$. Then, after a base change, there is a family $\epsilon: \mathcal{Y}  \to \PP^1\times \Delta$ of stable maps as in Proposition \ref{lem-unique}. 
If $\mathcal{Y}^*$ is not isomorphic to $\mathcal{B}^*$, we derive a contradiction arguing as in Proposition \ref{uniqueness-AssumpII}. 
So assume $\mathcal{Y}^* \cong \mathcal{B}^*$. Replacing $\Delta$ by a smooth curve mapping into $\Delta$ and passing through both $0$ and a sufficiently general $t$ we may take the base $\Delta$ to be one-dimensional and smooth. Let $\widetilde{\mathcal{B}}$ denote the family of stable curves obtained by stabilization from $\mathcal{B}$.

 We have two stabilization morphisms $r_1: \mathcal{B} \to \widetilde{\mathcal{B}}$ and $r_2: \mathcal{Y} \to  \widetilde{\mathcal{B}}$, which are isomorphisms outside of 
 the special fibre. Choose general points $a, b \in \PP^1$. Define $Z_a:=r_1(\phi^{-1}(a \times \Delta))$ and $Z_b:=r_2(\epsilon^{-1}(b \times \Delta))$.  Let $Z:=Z_a \cup Z_b$. Then $Z$ defines a flat family of $0$-dimensional subschemes of the fibres $\widetilde{B}_t$. For general $t$, the fibres $Z_{t}$ is a Cartier divisor on $\widetilde{B}_t$, whereas $Z_0 \seq \widetilde{B}_0$ defines a closed subscheme of the one-nodal curve $\widetilde{B}_0$. As $t \to 0$, the line bundles $\mathcal{O}_{B_t}(Z_t)$ tend to the torsion-free sheaf
 $$\mathcal{O}_{B_0}(Z_0):=\text{Hom}(I_{Z_0}, \mathcal{O}_{Z_0}).$$
 Furthermore, the conclusion of Proposition \ref{lem-unique} shows that $Z_{a,t}$ is not linearly equivalent to $Z_{b,t}$ for general $t$.  As we have previously seen, the base-point free pencil trick implies $h^0(\mathcal{O}_{B_t}(Z_t)) \geq 4$, and so $h^0(\mathcal{O}_{B_0}(Z_0)) \geq 4$ by semicontinuity.

 We can describe $Z_0$ as the union of the Cartier divisor
$$f^{-1}(a)+f^{-1}(b)+S' \seq C \setminus \{x_{m+1}, y_{m+1} \} \seq \widetilde{B}_0$$
with a closed subscheme $T$ of length two supported at the node $q \in \widetilde{B}_0$.

\vskip 3pt

Following \cite[\S 2]{coppens-free} we let $Z'_0 \seq Z_0$ denote the \emph{free part} of $Z_0$ (also see \cite[\S 2]{rosa-stohr}). Then
$$\mathcal{O}_{B_0}(Z'_0) := \text{Im}\Bigl\{H^0(\mathcal{O}_{B_0}(Z_0)) \otimes \mathcal{O}_{B_0} \to \mathcal{O}_{B_0}(Z_0) \Bigr\}$$
is globally generated, but not necessarily base-point free at the node $q \in \widetilde{B}_0$. Note that $Z_0'$ is the union of the Cartier divisor
 $$f^{-1}(a)+f^{-1}(b)+S'',$$
 where $S''$ is the complement of the base locus in $S'$, with a closed subscheme $T' \seq T$ of length at most two supported at the node.

Firstly, suppose the torsion-free sheaf $\mathcal{O}_{B_0}(Z'_0)$ is locally free. We have the partial normalization map $\mu_{m+1}: C \to B_0$. Any divisor in $|\mu_{m+1}^*(\mathcal{O}_{B_0}(Z'_0))|$ which contains one of the points $x_{m+1}, y_{m+1}$, must also contain the other point. Thus $\mu^*_{m+1}(Z'_0)$ is either $f^{-1}(a)+f^{-1}(b)+S''$ or $f^{-1}(a)+f^{-1}(b)+S''+x_{m+1}+y_{m+1}$. Either way, we get a contradiction to the hypothesis.

\vskip 4pt

 Next, suppose $\mathcal{O}_{B_0}(Z'_0)$ is not locally free. In this case, $\mathcal{O}_{B_0}(Z'_0)=(\mu_{m+1})_*(N)$ for some line bundle $N$ with $h^0(C,N)=h^0(\mathcal{O}_{B_0}(Z'_0))$ and $\deg(N)=\deg \mathcal{O}_{B_0}(Z'_0)-1$, see e.g.\ \cite[\S 10]{oda-seshadri}. The linear system $|N|$ contains a divisor of the form $f^{-1}(a)+f^{-1}(b)+S''+D$ where $D$ is supported on $x_{m+1}, y_{m+1}$ and has degree at most one. Once again, we get a contradiction to the hypothesis.\medskip

\end{proof}

We now bootstrap on Proposition \ref{1-node-special-case} in order to prove the induction step for Theorem \ref{hard-thm}.

\vskip 3pt

\begin{proof}[Proof of Theorem \ref{hard-thm}]
It remains to prove the induction step. Let $C$ be an integral nodal curve of genus $2a-2-m$ with a unique pencil $f:C\rightarrow \PP^1$ of degree $a-m-1$ and points $x_i,y_i\in C$ for $i=1, \ldots, m+1$ as in the hypotheses of the theorem for $n=m+1$. Set $A:=f^*\mathcal{O}_{\PP^1}(1)$, hence $h^0(C, A)=2$ and $\mbox{dim } W^1_{a-m-1}(C)=0$.


We claim $\omega_C\otimes A^{\vee}$ is base point free. Otherwise, there exists a  point $p \in C$ such that all sections of $\omega_C\otimes A^{\vee}$ vanish at $p$. Let $N:=\text{Ker}\{\mathrm{ev}_p: \omega_C \otimes A^{\vee} \to \mathbb C_p\}$. There is a partial normalization $\mu: \widetilde{C} \to C$ at $\delta\geq 0$ nodes and a line bundle $\widetilde{N}$ on $\widetilde{C}$ such that $\mu_*\widetilde{N}=N$. Then $\omega_{\widetilde{C}}\otimes \widetilde{N}^{\vee} \in W^2_{a-m-\delta}(\widetilde{C})$,  implying $\dim W^1_{a-m-1-\delta}(\widetilde{C}) \geq 1$, thus $\dim W^1_{a-m-1}(C) \geq 1$, a contradiction.

\vskip 3pt

We now argue exactly as in Proposition \ref{1-node-special-case}. Suppose  $[f,x_1,y_1, \ldots, x_{m+1}, y_{m+1}] \in Z_{m+1}$ which implies that $t=[f_{m+1},x_1,y_1, \ldots, x_m, y_m] \in Z_{m}$. Let $J \seq Z_m$ be any component containing $t$. By the same dimension count as in Proposition \ref{1-node-special-case}, the general point $[g: T\to \PP^1, (x'_i,y'_i)]$ of $J$ is once again a marked stable map with \emph{irreducible} source $T$. Moreover, from the conclusion of Proposition \ref{1-node-special-case}, we have the improved bound $$\dim Z_{m+1} \leq \dim \widetilde{\mathcal{M}}_{2a-2-m}^{\mathrm{ns}}(\PP^1,a-m-1;2(m+1))-1,$$
so that the locus of reducible curves in $J$ has codimension two about $t$. By Lemma \ref{lemma-first-part-thm} and Proposition \ref{uniqueness-AssumpII}, applied to $\Delta=J$, we derive a contradiction to the induction hypothesis.
 \end{proof}

\end{document}